\documentclass[ opre,nonblindrev]{informs3} 

\OneAndAHalfSpacedXI 

\usepackage{endnotes}
\let\footnote=\endnote

\usepackage{natbib}
 \bibpunct[, ]{(}{)}{,}{a}{}{,}%
 \def\bibfont{\small}%
\TheoremsNumberedThrough     
\ECRepeatTheorems

\EquationsNumberedThrough    

\usepackage{amsmath,amssymb} 
\usepackage{multirow}
\usepackage{authblk}
\usepackage{cite}
\usepackage{url}
\usepackage{pgf}
\usepackage{color}
\usepackage{booktabs} 
\usepackage{algorithm}
\usepackage[noend]{algorithmic}
\usepackage{pgfplots}
\usepackage{caption}
\usepackage{soul}
\usepackage{bbm}
\usepackage{bm}
\usepackage{pdflscape}
\usetikzlibrary{plotmarks}
\usepackage{enumitem}

\usepackage[absolute,overlay,showboxes]{textpos}
\usepackage{epstopdf}
\usepackage{subcaption}
\usepackage{changepage}
\usepackage{hyperref}

\usepackage{float}
\usepackage{tikz,pgfplots}
\usetikzlibrary{matrix}
\usepgfplotslibrary{groupplots}
\pgfplotsset{compat=newest}
\usetikzlibrary{calc,positioning,shapes,fit,arrows,shadows}
\tikzstyle{line} = [ draw, -latex']
\usetikzlibrary{shapes,arrows}
\usepgfplotslibrary{units}
\usepackage{url}
\usepackage{color}

\usetikzlibrary{shapes.geometric}
\usetikzlibrary{topaths,calc}

\renewcommand{\theARTICLETOP}{}
\renewcommand{\theLRHSecondLine}{}
\renewcommand{\theRRHSecondLine}{}
\renewcommand{\LRHFirstLine}{}
\renewcommand{\RRHFirstLine}{}
\renewcommand{\ECLRHFirstLine}{}
\renewcommand{\ECRRHFirstLine}{}

{\theoremheaderfont{THkey}\newtheorem{proof}{XXXXX}}

\usepackage{appendix}

\definecolor{darkgreen}{rgb}{0.1, 0.8, 0.1}

\begin{document}


\RUNAUTHOR{Tuncer and Kocuk}

\RUNTITLE{Risk-Averse Antibiotics Time Machine Problem}

\TITLE{Risk-Averse Antibiotics Time Machine Problem}

\ARTICLEAUTHORS{%
\AUTHOR{Deniz Tuncer}
\AFF{Industrial Engineering Program, Sabancı University, 34956 Istanbul, Turkey, \EMAIL{\href{mailto:dtuncer@sabanciuniv.edu}{dtuncer@sabanciuniv.edu}}} 
\AUTHOR{Burak Kocuk}
\AFF{Industrial Engineering Program, Sabancı University, 34956 Istanbul, Turkey, \EMAIL{\href{mailto:burak.kocuk@sabanciuniv.edu}{burak.kocuk@sabanciuniv.edu}}}
}

\ABSTRACT{Antibiotic resistance, which is a serious healthcare issue, emerges due to uncontrolled and repeated antibiotic use that causes bacteria to mutate and develop resistance to antibiotics. The Antibiotics Time Machine Problem aims to come up with treatment plans that maximize the probability of reversing these mutations. Motivated by the severity of the problem, we develop a risk-averse approach and formulate a scenario-based mixed-integer linear program with a conditional value-at-risk objective function. We propose a risk-averse scenario batch decomposition algorithm that partitions the scenarios into manageable risk-averse subproblems, enabling the construction of lower and upper bounds. We develop several algorithmic enhancements in the form of stronger no-good cuts and symmetry breaking constraints in addition to scenario regrouping and warm starting. We conduct extensive computational experiments for static and dynamic versions of the problem on a real dataset and demonstrate the effectiveness of our approach. Our results suggest that risk-averse solutions can achieve significantly better worst-case performance compared to risk-neutral solutions with a slight decrease in terms of the average performance, especially for the dynamic version. Although our methodology is presented in the context of the Antibiotics Time Machine Problem, it can be adapted to other risk-averse problem settings in which the decision variables come from special ordered sets of type~one.}


\maketitle

%


\section{Introduction and Literature Review}
\label{sec:Intro}

Antibiotic resistance is a serious concern in modern medicine. Misuse or successive administration of antibiotics may cause mutations of the bacteria, which facilitate the resistance to the known antibiotics. Due to their immunization to the antibiotics, it may not be possible to cure the diseases caused by the antibiotic-resistant bacteria. 
2.8 million antibiotic-resistant infections  occurred in the US, and more than 35 thousand people lost their lives in 2019 \citep{reportglobalresistance} due to antibiotic resistance. A study on Chinese healthcare system \citep{chinesestudy} shows that about 600 thousand deaths have occurred due to bacterial antimicrobial resistance. The same study also shows that in China, 173 million unreasonable prescriptions of antibiotics are given each year, which also contributes to the antibiotic resistance. Adverse effects of some overprescribed antibiotics, such as a widely used antibiotic, Amoxicillin-clavulanic acid is shown to cause life threatening adverse effects \citep{Chang2007}, which further suggests that antibiotic administration should be done with utmost care.

The antibiotics that humankind have developed during the history are in danger of becoming obsolete. Therefore, we need new strategies to combat antibiotic resistance. While developing new antibiotics seem to be the obvious solution, research and development process of a new antibiotic is quite costly. Even if the necessary funding is found, the process itself requires extensive human resource and the time it takes to develop and start using the drug is unpredictable \citep{Ventola2015}. A recent study by \citet{Liu2024} shows that only 18 novel antibiotics have been approved since 2014. Since development of new antibiotics is unlikely to happen soon, as an interim solution, researchers have proposed the method of repurposing drugs. Drug repurposing aims to break the antibiotic resistance by co-administering the antibiotics with non-antibiotic drugs \citep{Brown2015}. This repurposing idea may increase the lifespan of the existing antibiotics while the new antibiotics are being developed. While drug repurposing seems to be the recipe for slowing down the antibiotic resistance, there exists concerns about the future efficiency of the approach, since it may cause further antibiotic resistance that we may not be able to cure \citep{Talat2022}.

{In search of a remedy for the antibiotic resistance, some of the recent works on the problem focus on artificial intelligence (AI) and deep learning-based methods \citep{Liu2024}. The study by \citet{Liu2023} aims to create new antibiotics by utilizing deep learning methods to develop molecules that have activity against the Acinetobacter baumannii bacterium.  \citep{Wong2023} construct explainable graphs using deep learning methods to predict the compounds that might have high antibiotic activity. Apart from drug development, AI and deep learning methods are utilized to predict the antibiotic resistance of the bacteria that cause the infections \citep{Zagajewski2023,Ding2024}. These methods are also utilized for identifying the drugs that may be eligible for drug repurposing \citep{Aggarwal2024}, which can be a future research direction for use of AI for antibiotic resistance problem.}

{It is found out that when some antibiotics are applied sequentially, their cross effect may change the probability of the bacteria developing antibiotic resistance, called \textit{collateral sensitivity} \citep{imamovic2013,nichol2015steering}. Although there are opposing ideas  that drug sequencing is unlikely to prevent antibiotic resistance \citep{bergstrom2004ecological,vanduijn2018}, some studies claim that the collateral sensitivity can eliminate antibiotic resistance \citep{Tyers2019,Maltas2019,Maltas2021}. Another study \citep{Aulin2021} shows that the order of drug administration is crucial to the efficiency of the collateral sensitivity based treatments.}

A further complicating factor in these studies is the issue of quantifying the antibiotic resistance phenomenon \citep{mira2015rational}. Typically, the effect of an antibiotic can be modeled as a Markov chain, where the states correspond to the genotypes of the bacterium identified by a string of alleles and probabilities are computed as a function of the growth rates measurements of these genotypes under the administration of the antibiotic. In addition to the stochasticity of an antibiotic administration, another source of randomness arises due to the fact that the growth rates measurements, and consequently,  the transition probability  matrices governing the Markov chain can be uncertain as well \citep{mira2017statistical}. 
%
%

The Antibiotics Time Machine Problem  is an NP-Hard combinatorial optimization problem \citep{tran2017antibiotics} that aims to find the best   treatment plan of finite length using a set of given antibiotics so that the probability of reversing the antibiotic resistance is maximized. The first paper to tackle this problem is \citet{mira2015rational}, in which the authors use a complete enumeration technique to find an optimal sequence. Recently, a compact-size mixed-integer linear programming (MILP) formulation is proposed in \citet{kocuk2022optimization} that works much more efficiently. We will refer to the version studied in these three papers as \textit{deterministic} since the probability matrices are computed using a single growth rate measurement and \textit{static} since the  antibiotic sequence is determined at the beginning of the treatment.
Later, the work of \citet{kocuk2022optimization} is extended to the cases of \textit{stochastic} and  \textit{dynamic} versions of the Antibiotics Time Machine Problem  in \citet{mesum2024}. Here, the stochastic version refers to the case where multiple growth rate measurements are available and the probability matrices are modeled as random variables whereas the dynamic version refers to the decision policies in which the treatment plan is adaptive to the observed state transitions. 
It is proven in \citet{mesum2024}
  that the stochastic static version of the problem can be solved as a special instance of the  deterministic static version, which can be solved as an MILP. On the other hand, the stochastic dynamic version of the problem can be solved as a special instance of the  deterministic dynamic version, which can be solved using a dynamic program (DP). 

Although the work by \citet{mesum2024} presents the first attempt to solve the   Antibiotics Time Machine Problem under uncertainty, it assumes that the objective function is the maximization of the \textit{expected} probability of reversing the antibiotic resistance. However, this \textit{risk-neutral} approach may not be suitable for a healthcare application due to its sensitive nature. Therefore, in this paper, we study a \textit{risk-averse} approach, in which we maximize the expected probability of the worst $\alpha$ fraction of the realizations akin to a Conditional-Value-at-Risk (CVaR)-type  objective function. As opposed to the expectation, which is a linear function, CVaR is not a linear function, therefore, the theoretical results from \citet{mesum2024} are not applicable. It turns out that both   static and dynamic versions of the risk-averse Antibiotics Time Machine Problem are quite challenging computationally. 

In this paper, we study scenario-based MILP formulations for both static and dynamic versions of the risk-averse Antibiotics Time Machine Problem for the first time. Since these MILPs do not scale well with the number of scenarios, we develop a risk-averse scenario batch decomposition algorithm, originally developed for risk-neutral  \citep{ahmed2013decomposition} and risk-averse problems \citep{deng2018}. As opposed to these papers, which utilize single-scenario subproblems, we use batch scenario subproblems in order to obtain stronger dual bounds. To further improve the computational efficiency of the standard algorithm, we derive two types of  no-good cuts that exploit problem structure and solution symmetry. In addition, we introduce scenario regrouping and warm start techniques that help accelerate the decomposition algorithm. Although these contributions to the literature are presented in the context of the Antibiotics Time Machine Problem, most of them are applicable for risk-averse optimization problems in which the decision variables come from special ordered sets  of type one. 

Our extensive computational experiments conducted using real growth rate measurements \citep{mira2017statistical} give interesting insights. As expected, switching from the risk-neutral solution to the risk-averse solution brings considerable gains from the worst case realizations albeit a loss from the average performance. Interestingly, as the treatment length increases, the amount of gain from the worst case realizations starts to dominate the amount of loss from the average performance. In addition, the benefit of the risk-averse approach is even more pronounced for the dynamic version for which solutions with almost negligible sacrifice from the average performance may increase the worst case realizations considerably.

The rest of the paper is organized as follows: In Section~\ref{sec:ProblemStatement}, we introduce  the deterministic Antibiotics Time Machine Problem. In Sections~\ref{sec:ABRStatic} and~\ref{sec:ABR-dynamic}, we study   static and dynamic versions of the risk-averse Antibiotics Time Machine Problem, respectively. We propose  decomposition algorithms for their MILP formulations and provide several algorithmic enhancements.   In Section~\ref{sec:compresults}, we present the results of our extensive computational experiments. Finally, we conclude our paper   in Section~\ref{sec:conclusions}.

\section{Deterministic Problem}
\label{sec:ProblemStatement}

In this section, we present the deterministic version of the Antibiotics Time Machine Problem studied in \citet{mira2015rational,kocuk2022optimization} as this will make it easier to explain the stochastic versions later on.
Suppose that we are interested in a bacterium with $a$ many alleles. We  denote each  genotype by the vector $\bm{g} \in \{0,1\}^a$, where $g_l=0$ (resp. $g_l=1$) indicates an unmutated (resp. mutated) allele $l$. We call the special genotype associated with the vector $\bm{g}=0$ as the ``wild type", which is the desired state where the bacterium is sensitive to an antibiotic treatment. 

Suppose we have $K$ many antibiotics at hand, each of which comes with a probability transition matrix $ \bm{M^k} \in \mathbb{R}^{d \times d} $, $k=1,\dots,K$, where $d=2^a$
(we explain how these matrices are obtained from experimental data  in Appendix~\ref{app:matrixConstruction}). In particular,  the probability of transitioning from genotype~$i$ to genotype~$j$ when antibiotic~$k$ is applied is denoted by ${M_{ij}^k}$. 
%
%
Due to the generally accepted principle of Strong Selection Weak Mutation (see, e.g., \citet{gillespie1983simple}, \citet{gillespie1984molecular}), we will assume that at most one allele can change in each transition. In other words, the probability of a transition between genotype $\bm{g}$ and genotype $\bm{g'}$ is zero if $\| \bm{g}-\bm{g'}\|_1 > 1 $. 

Let us   illustrate the setup with an example in Figure~\ref{fig:abr-image} in which there are $d=2^3$ genotypes and $K=2$  antibiotics represented by different types of arcs. An arc indicates that there is a positive probability of going from a state to another under that antibiotic. For example, starting from genotype 010, if we apply antibiotic 1 (solid line), the stochastic process can go to the genotype 000 or genotype 011. If the current state is 101 and we apply antibiotic 2 (dashed line), the next state would be 111.

\tikzset{every picture/.style={line width=0.75pt}} 
\begin{figure}[H]
\centering

\caption{Illustration of an Antibiotics Time Machine Problem instance with $a=3$ alleles and $K=2$ antibiotics.} \label{fig:abr-image}

\tikzset{every picture/.style={line width=0.75pt}} 

\begin{tikzpicture}[x=0.675pt,y=0.675pt,yscale=-1,xscale=1]

\draw   (188,136) .. controls (188,122.19) and (199.19,111) .. (213,111) .. controls (226.81,111) and (238,122.19) .. (238,136) .. controls (238,149.81) and (226.81,161) .. (213,161) .. controls (199.19,161) and (188,149.81) .. (188,136) -- cycle ;
\draw   (268,77) .. controls (268,63.19) and (279.19,52) .. (293,52) .. controls (306.81,52) and (318,63.19) .. (318,77) .. controls (318,90.81) and (306.81,102) .. (293,102) .. controls (279.19,102) and (268,90.81) .. (268,77) -- cycle ;
\draw   (268,137) .. controls (268,123.19) and (279.19,112) .. (293,112) .. controls (306.81,112) and (318,123.19) .. (318,137) .. controls (318,150.81) and (306.81,162) .. (293,162) .. controls (279.19,162) and (268,150.81) .. (268,137) -- cycle ;
\draw   (268,197) .. controls (268,183.19) and (279.19,172) .. (293,172) .. controls (306.81,172) and (318,183.19) .. (318,197) .. controls (318,210.81) and (306.81,222) .. (293,222) .. controls (279.19,222) and (268,210.81) .. (268,197) -- cycle ;
\draw   (348,76) .. controls (348,62.19) and (359.19,51) .. (373,51) .. controls (386.81,51) and (398,62.19) .. (398,76) .. controls (398,89.81) and (386.81,101) .. (373,101) .. controls (359.19,101) and (348,89.81) .. (348,76) -- cycle ;
\draw   (348,136) .. controls (348,122.19) and (359.19,111) .. (373,111) .. controls (386.81,111) and (398,122.19) .. (398,136) .. controls (398,149.81) and (386.81,161) .. (373,161) .. controls (359.19,161) and (348,149.81) .. (348,136) -- cycle ;
\draw   (348,196) .. controls (348,182.19) and (359.19,171) .. (373,171) .. controls (386.81,171) and (398,182.19) .. (398,196) .. controls (398,209.81) and (386.81,221) .. (373,221) .. controls (359.19,221) and (348,209.81) .. (348,196) -- cycle ;
\draw   (428,136) .. controls (428,122.19) and (439.19,111) .. (453,111) .. controls (466.81,111) and (478,122.19) .. (478,136) .. controls (478,149.81) and (466.81,161) .. (453,161) .. controls (439.19,161) and (428,149.81) .. (428,136) -- cycle ;
\draw [color={rgb, 255:red, 208; green, 2; blue, 27 }  ,draw opacity=1 ]   (226.5,114.75) -- (265.92,83.98) ;
\draw [shift={(267.5,82.75)}, rotate = 142.03] [color={rgb, 255:red, 208; green, 2; blue, 27 }  ,draw opacity=1 ][line width=0.75]    (10.93,-3.29) .. controls (6.95,-1.4) and (3.31,-0.3) .. (0,0) .. controls (3.31,0.3) and (6.95,1.4) .. (10.93,3.29)   ;
\draw [color={rgb, 255:red, 208; green, 2; blue, 27 }  ,draw opacity=1 ]   (318.17,67.33) -- (346.67,67.33) ;
\draw [shift={(348.67,67.33)}, rotate = 180] [color={rgb, 255:red, 208; green, 2; blue, 27 }  ,draw opacity=1 ][line width=0.75]    (10.93,-3.29) .. controls (6.95,-1.4) and (3.31,-0.3) .. (0,0) .. controls (3.31,0.3) and (6.95,1.4) .. (10.93,3.29)   ;
\draw [color={rgb, 255:red, 208; green, 2; blue, 27 }  ,draw opacity=1 ]   (433,154.25) -- (397.56,182.5) ;
\draw [shift={(396,183.75)}, rotate = 321.43] [color={rgb, 255:red, 208; green, 2; blue, 27 }  ,draw opacity=1 ][line width=0.75]    (10.93,-3.29) .. controls (6.95,-1.4) and (3.31,-0.3) .. (0,0) .. controls (3.31,0.3) and (6.95,1.4) .. (10.93,3.29)   ;
\draw [color={rgb, 255:red, 208; green, 2; blue, 27 }  ,draw opacity=1 ]   (439.5,113.25) -- (399.58,82.47) ;
\draw [shift={(398,81.25)}, rotate = 37.64] [color={rgb, 255:red, 208; green, 2; blue, 27 }  ,draw opacity=1 ][line width=0.75]    (10.93,-3.29) .. controls (6.95,-1.4) and (3.31,-0.3) .. (0,0) .. controls (3.31,0.3) and (6.95,1.4) .. (10.93,3.29)   ;
\draw [color={rgb, 255:red, 208; green, 2; blue, 27 }  ,draw opacity=1 ]   (399.17,130.25) -- (425.67,130.25) ;
\draw [shift={(427.67,130.25)}, rotate = 180] [color={rgb, 255:red, 208; green, 2; blue, 27 }  ,draw opacity=1 ][line width=0.75]    (10.93,-3.29) .. controls (6.95,-1.4) and (3.31,-0.3) .. (0,0) .. controls (3.31,0.3) and (6.95,1.4) .. (10.93,3.29)   ;
\draw [color={rgb, 255:red, 208; green, 2; blue, 27 }  ,draw opacity=1 ]   (268,128.42) -- (238.5,128.42) ;
\draw [shift={(236.5,128.42)}, rotate = 360] [color={rgb, 255:red, 208; green, 2; blue, 27 }  ,draw opacity=1 ][line width=0.75]    (10.93,-3.29) .. controls (6.95,-1.4) and (3.31,-0.3) .. (0,0) .. controls (3.31,0.3) and (6.95,1.4) .. (10.93,3.29)   ;
\draw [color={rgb, 255:red, 208; green, 2; blue, 27 }  ,draw opacity=1 ]   (269.81,184.19) -- (228,157.25) ;
\draw [shift={(226,156)}, rotate = 33.0] [color={rgb, 255:red, 208; green, 2; blue, 27 }  ,draw opacity=1 ][line width=0.75]    (10.93,-3.29) .. controls (6.95,-1.4) and (3.31,-0.3) .. (0,0) .. controls (3.31,0.3) and (6.95,1.4) .. (10.93,3.29)   ;
\draw [color={rgb, 255:red, 208; green, 2; blue, 27 }  ,draw opacity=1 ]   (386,52.75) .. controls (403.73,15.32) and (347.24,15.73) .. (362.74,50.63) ;
\draw [shift={(363.5,52.25)}, rotate = 249.33] [color={rgb, 255:red, 208; green, 2; blue, 27 }  ,draw opacity=1 ][line width=0.75]    (10.93,-3.29) .. controls (6.95,-1.4) and (3.31,-0.3) .. (0,0) .. controls (3.31,0.3) and (6.95,1.4) .. (10.93,3.29)   ;
\draw [color={rgb, 255:red, 208; green, 2; blue, 27 }  ,draw opacity=1 ]   (348.5,188.75) -- (319,188.75) ;
\draw [shift={(317,188.75)}, rotate = 360] [color={rgb, 255:red, 208; green, 2; blue, 27 }  ,draw opacity=1 ][line width=0.75]    (10.93,-3.29) .. controls (6.95,-1.4) and (3.31,-0.3) .. (0,0) .. controls (3.31,0.3) and (6.95,1.4) .. (10.93,3.29)   ;

\draw [color={rgb, 255:red, 126; green, 211; blue, 33 }  ,draw opacity=1 ] [dash pattern={on 4.5pt off 4.5pt}]  (385.17,219.42) .. controls (397.42,245.88) and (352.04,246.89) .. (361.99,220.56) ;
\draw [shift={(362.67,218.92)}, rotate = 102.09] [color={rgb, 255:red, 126; green, 211; blue, 33 }  ,draw opacity=1 ][line width=0.75]    (10.93,-3.29) .. controls (6.95,-1.4) and (3.31,-0.3) .. (0,0) .. controls (3.31,0.3) and (6.95,1.4) .. (10.93,3.29)   ;
\draw [color={rgb, 255:red, 126; green, 211; blue, 33 }  ,draw opacity=1 ] [dash pattern={on 4.5pt off 4.5pt}]  (317,202.83) -- (347,202.83) ;
\draw [shift={(349,202.83)}, rotate = 180] [color={rgb, 255:red, 126; green, 211; blue, 33 }  ,draw opacity=1 ][line width=0.75]    (10.93,-3.29) .. controls (6.95,-1.4) and (3.31,-0.3) .. (0,0) .. controls (3.31,0.3) and (6.95,1.4) .. (10.93,3.29)   ;
\draw [color={rgb, 255:red, 126; green, 211; blue, 33 }  ,draw opacity=1 ] [dash pattern={on 4.5pt off 4.5pt}]  (265.67,192.17) -- (221.34,163.26) ;
\draw [shift={(219.67,162.17)}, rotate = 33.11] [color={rgb, 255:red, 126; green, 211; blue, 33 }  ,draw opacity=1 ][line width=0.75]    (10.93,-3.29) .. controls (6.95,-1.4) and (3.31,-0.3) .. (0,0) .. controls (3.31,0.3) and (6.95,1.4) .. (10.93,3.29)   ;
\draw [color={rgb, 255:red, 126; green, 211; blue, 33 }  ,draw opacity=1 ] [dash pattern={on 4.5pt off 4.5pt}]  (307.67,96.83) -- (348.26,120.01) ;
\draw [shift={(350,121)}, rotate = 209.72] [color={rgb, 255:red, 126; green, 211; blue, 33 }  ,draw opacity=1 ][line width=0.75]    (10.93,-3.29) .. controls (6.95,-1.4) and (3.31,-0.3) .. (0,0) .. controls (3.31,0.3) and (6.95,1.4) .. (10.93,3.29)   ;
\draw [color={rgb, 255:red, 126; green, 211; blue, 33 }  ,draw opacity=1 ] [dash pattern={on 4.5pt off 4.5pt}]  (348.33,80.17) -- (320.33,80.17) ;
\draw [shift={(318.33,80.17)}, rotate = 360] [color={rgb, 255:red, 126; green, 211; blue, 33 }  ,draw opacity=1 ][line width=0.75]    (10.93,-3.29) .. controls (6.95,-1.4) and (3.31,-0.3) .. (0,0) .. controls (3.31,0.3) and (6.95,1.4) .. (10.93,3.29)   ;
\draw [color={rgb, 255:red, 126; green, 211; blue, 33 }  ,draw opacity=1 ] [dash pattern={on 4.5pt off 4.5pt}]  (315,148.17) -- (350.8,178.21) ;
\draw [shift={(352.33,179.5)}, rotate = 220.01] [color={rgb, 255:red, 126; green, 211; blue, 33 }  ,draw opacity=1 ][line width=0.75]    (10.93,-3.29) .. controls (6.95,-1.4) and (3.31,-0.3) .. (0,0) .. controls (3.31,0.3) and (6.95,1.4) .. (10.93,3.29)   ;
\draw [color={rgb, 255:red, 126; green, 211; blue, 33 }  ,draw opacity=1 ] [dash pattern={on 4.5pt off 4.5pt}]  (238.33,138.17) -- (267,138.17) ;
\draw [shift={(269,138.17)}, rotate = 180] [color={rgb, 255:red, 126; green, 211; blue, 33 }  ,draw opacity=1 ][line width=0.75]    (10.93,-3.29) .. controls (6.95,-1.4) and (3.31,-0.3) .. (0,0) .. controls (3.31,0.3) and (6.95,1.4) .. (10.93,3.29)   ;
\draw [color={rgb, 255:red, 126; green, 211; blue, 33 }  ,draw opacity=1 ] [dash pattern={on 4.5pt off 4.5pt}]  (271,92.17) -- (235.27,118.97) ;
\draw [shift={(233.67,120.17)}, rotate = 323.13] [color={rgb, 255:red, 126; green, 211; blue, 33 }  ,draw opacity=1 ][line width=0.75]    (10.93,-3.29) .. controls (6.95,-1.4) and (3.31,-0.3) .. (0,0) .. controls (3.31,0.3) and (6.95,1.4) .. (10.93,3.29)   ;
\draw [color={rgb, 255:red, 126; green, 211; blue, 33 }  ,draw opacity=1 ] [dash pattern={on 4.5pt off 4.5pt}]  (439.67,158.83) -- (400.56,190.25) ;
\draw [shift={(399,191.5)}, rotate = 321.23] [color={rgb, 255:red, 126; green, 211; blue, 33 }  ,draw opacity=1 ][line width=0.75]    (10.93,-3.29) .. controls (6.95,-1.4) and (3.31,-0.3) .. (0,0) .. controls (3.31,0.3) and (6.95,1.4) .. (10.93,3.29)   ;
\draw [color={rgb, 255:red, 126; green, 211; blue, 33 }  ,draw opacity=1 ] [dash pattern={on 4.5pt off 4.5pt}]  (397.67,141.5) -- (425.67,141.5) ;
\draw [shift={(427.67,141.5)}, rotate = 180] [color={rgb, 255:red, 126; green, 211; blue, 33 }  ,draw opacity=1 ][line width=0.75]    (10.93,-3.29) .. controls (6.95,-1.4) and (3.31,-0.3) .. (0,0) .. controls (3.31,0.3) and (6.95,1.4) .. (10.93,3.29)   ;
\draw [color={rgb, 255:red, 208; green, 2; blue, 27 }  ,draw opacity=1 ]   (357.67,114.83) -- (315.74,90.91) ;
\draw [shift={(314,89.92)}, rotate = 29.71] [color={rgb, 255:red, 208; green, 2; blue, 27 }  ,draw opacity=1 ][line width=0.75]    (10.93,-3.29) .. controls (6.95,-1.4) and (3.31,-0.3) .. (0,0) .. controls (3.31,0.3) and (6.95,1.4) .. (10.93,3.29)   ;
\draw [color={rgb, 255:red, 208; green, 2; blue, 27 }  ,draw opacity=1 ]   (349.67,147.5) -- (311.6,176.05) ;
\draw [shift={(310,177.25)}, rotate = 323.13] [color={rgb, 255:red, 208; green, 2; blue, 27 }  ,draw opacity=1 ][line width=0.75]    (10.93,-3.29) .. controls (6.95,-1.4) and (3.31,-0.3) .. (0,0) .. controls (3.31,0.3) and (6.95,1.4) .. (10.93,3.29)   ;
\draw [color={rgb, 255:red, 126; green, 211; blue, 33 }  ,draw opacity=1 ] [dash pattern={on 4.5pt off 4.5pt}]  (394.33,90.83) -- (430.74,118.3) ;
\draw [shift={(432.33,119.5)}, rotate = 217.03] [color={rgb, 255:red, 126; green, 211; blue, 33 }  ,draw opacity=1 ][line width=0.75]    (10.93,-3.29) .. controls (6.95,-1.4) and (3.31,-0.3) .. (0,0) .. controls (3.31,0.3) and (6.95,1.4) .. (10.93,3.29)   ;
\draw [color={rgb, 255:red, 208; green, 2; blue, 27 }  ,draw opacity=1 ]   (313.4,121.4) -- (348.71,89.93) ;
\draw [shift={(350.2,88.6)}, rotate = 138.29] [color={rgb, 255:red, 208; green, 2; blue, 27 }  ,draw opacity=1 ][line width=0.75]    (10.93,-3.29) .. controls (6.95,-1.4) and (3.31,-0.3) .. (0,0) .. controls (3.31,0.3) and (6.95,1.4) .. (10.93,3.29)   ;

\draw (200,130) node [anchor=north west][inner sep=0.75pt]   [align=left] {000};
\draw (280,70) node [anchor=north west][inner sep=0.75pt]   [align=left] {001};
\draw (280,130) node [anchor=north west][inner sep=0.75pt]   [align=left] {010};
\draw (280,190) node [anchor=north west][inner sep=0.75pt]   [align=left] {100};
\draw (360,70) node [anchor=north west][inner sep=0.75pt]   [align=left] {011};
\draw (360,130) node [anchor=north west][inner sep=0.75pt]   [align=left] {101};
\draw (360,190) node [anchor=north west][inner sep=0.75pt]   [align=left] {110};
\draw (440,130) node [anchor=north west][inner sep=0.75pt]   [align=left] {111};
\end{tikzpicture}
\end{figure}
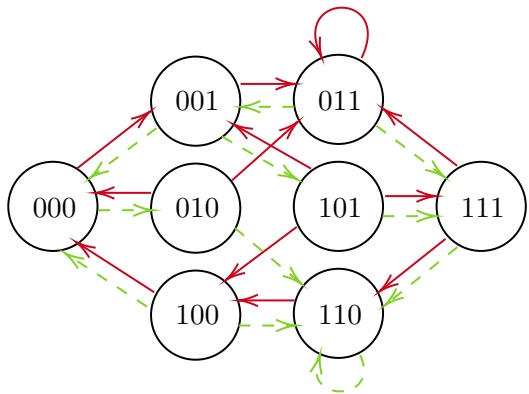

\vspace{-1cm}

The goal in the Antibiotics Time Machine Problem is to come up with a treatment plan that maximizes the probability of reaching the wild type starting from a given initial genotype with a predetermined number $(N)$ of antibiotic applications. We will consider two versions of the problem depending on the type of plans allowed: static and dynamic. The solution of the static version is a sequence of antibiotics determined at the beginning of the planning horizon. 
For the instance in Figure~\ref{fig:abr-image}, an example sequence is 1-2-2, which  gives the probability of reaching the wild type    $M_{111,011}^1 \cdot M_{011,001}^2 \cdot M_{001,000}^2 $ if the initial genotype is 111.
In the dynamic version, the solution is expressed by a policy, where we decide which antibiotic to apply based on the genotype at each decision epoch. Note that the dynamic version requires intermediate observations of state transitions.  
For the instance in Figure~\ref{fig:abr-image}, an example   policy would be to apply antibiotic $k  = 1 + g_1$ if the current state is $\bm{g}$, which gives the probability of reaching the wild type
$M_{111,011}^1 \cdot M_{011,001}^2 \cdot M_{001,000}^2 + M_{111,110}^1 \cdot M_{110,100}^1 \cdot M_{100,000}^1  $ if the initial genotype is 111.

In the rest of the paper, we will consider the stochastic version of the Antibiotics Time Machine Problem in which the matrices $ \bm{M^k}$ are uncertain.
In particular, we will assume that 
$ \bm{M^k}$ is a discrete random variable that is equally likely to take value from the set of all possible antibiotic matrices $\{\bm{M^{k,h}} \in \mathbb{R}^{d \times d}:   h \in \mathcal{S} \}$, where $\mathcal{S}$ is an index set (see Appendix~\ref{app:matrixConstruction} for details). In reality, the cardinality of $\mathcal{S}$ can be quite large (e.g., $12^{16}$, see, \citet{mesum2024}). Instead, we will  employ Sample Average Approximation (SAA) and take a random sample $\mathcal{H} \subseteq \mathcal{S}^{K}$ such that $|\mathcal{H}| \ll |\mathcal{S}|^{K}$. 

\section{Static Stochastic Problem}
\label{sec:ABRStatic}


In this section, we will assume that we are given a stochastic Antibiotics Time Machine Problem instance with $d=2^a$ many alleles, $K$ many antibiotics and a scenario sample $\mathcal{H}$.  
We   present how to formulate and solve the risk-neutral and risk-averse versions of this problem in Sections~\ref{sec:ABR-riskneutral-stoc} and~\ref{sec:ABR-riskaverse-stoc}, respectively. Then, we propose  a scenario decomposition algorithm in Section~\ref{sec:scen-decomp-static} and several algorithmic enhancements to increase its computational efficiency in Section~\ref{sec:improvements-static}.

\subsection{Risk-Neutral Version}
\label{sec:ABR-riskneutral-stoc}


In this version of the problem, we assume that the objective function maximizes the \textit{expected probability} of reaching the wild type in $N$ steps given an initial genotype. In order to formulate this problem, we first  introduce some notation: Let $\bm{r} \in \mathbb{R}^d$ and $\bm{q}\in \mathbb{R}^d$ denote the standard unit vectors which correspond to the initial distribution and the desired distribution, respectively (we will use a genotype, a binary vector in $\{0,1\}^a$, and its ID, a positive integer in $\{1,\dots,2^a\}$, interchangeably throughout the paper). 
        We define a binary decision variable  $x_{n,k}$ that takes the value 1 if antibiotic~$k$ is applied at step $n$, and 0 otherwise.
Then, the static risk-neutral optimization model can be formulated as  
\begin{equation} \label{eq:stochasticSampledNeutral}
\max_{\bm{x} \in \mathcal{X}}   \frac1{|\mathcal{H}|} \sum_{h\in \mathcal{H}} 
\underbrace{\bm{r}^\top \bigg[ \prod_{n=1}^N \bigg(\sum_{k=1}^K\bm{M^{k,h}} x_{n,k}\bigg)\bigg] \bm{q}}_{:=f_h (\bm{x})} ,
\end{equation} 
where the set
$ \mathcal{X} := \{ \bm{x} \in \{0,1\}^{N \times K} : \sum_{k=1}^K x_{n,k}=1, n=1,\dots,N \}$ ensures that exactly one antibiotic~$k$ is chosen for each position $n$.
The objective function is calculated for each scenario $h$ depending on the antibiotic selection by multiplying the associated probability transition matrices, and then,  averaged over $h \in \mathcal{H}$.
Notice that the formulation \eqref{eq:stochasticSampledNeutral} is a nonlinear integer programming problem, which seems 
 difficult to solve  directly. Instead, we will introduce  new variables, using which the problem is reformulated as an MILP.
For this purpose, let $\bm{u_n^h}$ denote the probability distribution  after the $n$-th antibiotic is applied under scenario $h$. Then, our decision variables satisfy the following recursion:
\begin{equation}
\label{eq:static-recursion}
     \sum_{k=1}^{K} \bm{(u_{n-1}^h)}^\top \bm{M^{k,h}}x_{n,k} = \bm{(u_n^h)}^\top \quad n=1,\dots,N .
\end{equation}

Adapting the deterministic formulation in~\citet{kocuk2022optimization,mesum2024} to this setting, it is possible to reformulate problem~\eqref{eq:stochasticSampledNeutral} exactly as the following MILP:
\begin{subequations} \label{eq:stochasticSampledDecompNeutralMILP}
\begin{align}
\textup{S-RN($\mathcal{H}$):\quad}
\max_{\bm{u}, \, \bm{v}, \, \bm{x} \in \mathcal{X} } \  &  \frac{1}{|\mathcal{H}|} ~ \sum_{h \in \mathcal{H}} \bm{q}^\top \bm{u_N^h} \label{eq:riskneutralMILP-obj} \\
\textup{s.t.}  \ &  \bm{u_0^{h}}=\bm{r} &h& \in \mathcal{H} \label{eq:riskneutralMILP-consb} \\
 \  & \sum_{k=1}^K \bm{v_{n-1,k}^h} = \bm{u_{n-1}^h} &n&=1,\dots,N, h \in \mathcal{H} \label{eq:riskneutralMILP-consc}   \\ 
 \  &  \sum_{k=1}^K (\bm{v^{h}_{n-1,k}})^\top \bm{M^{k,h}}  = (\bm{u^{h}_n})^\top   &n&=1,\dots,N, h \in \mathcal{H} \label{eq:riskneutralMILP-consd} \\ 
 \  &  \ \bm{e}^\top \bm{v_{n-1,k}^h} = x_{n,k} &n&=1,\dots,N,  k=1,\dots,K, h \in \mathcal{H} \label{eq:riskneutralMILP-conse} \\
 \ &   \bm{u_{n}^h} \in \mathbb{R}_+^d , \  \bm{v_{n-1,k}^h} \in \mathbb{R}_+^d   &n&=1,\dots,N,  k=1,\dots,K, h \in \mathcal{H}. \label{eq:riskneutralMILP-consf}
\end{align}
\end{subequations}
In this formulation, the objective function \eqref{eq:riskneutralMILP-obj} maximizes the expected probability of reaching the desired state. Constraint~\eqref{eq:riskneutralMILP-consb} ensures that the initial state distribution is $\bm{r}$. Constraints \eqref{eq:riskneutralMILP-consc}-\eqref{eq:riskneutralMILP-conse} linearize the recursive relation \eqref{eq:static-recursion} between each drug application. Here, variables $\bm{v_n^h}$ are the copy variables that help linearizing this relation and $\bm{e}$ denotes the vector of ones.  Finally, constraint~\eqref{eq:riskneutralMILP-consf} gives the variable domain restrictions.


\subsection{Risk-Averse Version}
\label{sec:ABR-riskaverse-stoc}

While solving the problem with a risk-neutral approach that solely focuses on the average performance seems plausible at first,  it is important for a treatment to be successful even for the worst cases due to the sensitive nature of the Antibiotics Time Machine Problem.
Therefore, solving the problem in a risk-averse manner would be a more reasonable approach. Our risk measure is CVaR, where we aim to maximize the average of the worst $\alpha$ fraction of the realizations.

To motivate the relevance of risk-aversion, we compare 
the solutions obtained from the risk-neutral and risk-averse approaches with the initial genotype 0001
under $N=5$.
In this case, the average performances of risk-neutral and risk-averse solutions are 0.55 and 0.52, respectively, which are relatively close. However, the average of the worst 10\% of the realizations is 0.00 and 0.31 for risk-neutral and risk-averse solutions, respectively, resulting in a stark difference.
To better compare the two approaches, we also provide 
the histograms corresponding to the frequency of certain probability ranges of reaching the wild type in Figure~\ref{fig:riskneutral vs riskaverse}.  
We clearly see that the risk-neutral solution has a large number of realizations with low probability of reaching the wild type, which is significantly reduced in the risk-averse solution.
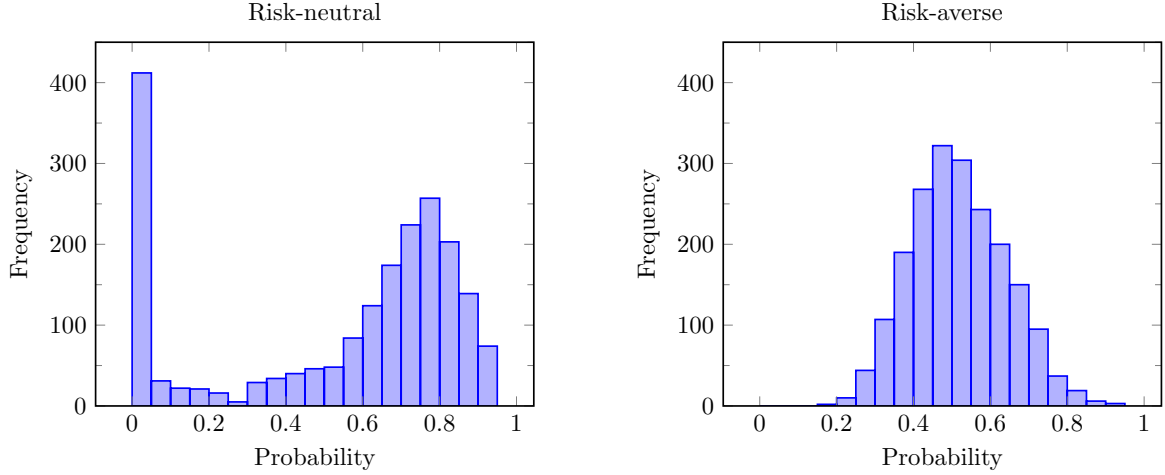
\begin{figure}[ht]
\caption{Comparison of risk-neutral and risk-averse solutions for the initial genotype 0001 and $N=5$.}
\label{fig:riskneutral vs riskaverse}
\begin{subfigure}{.5\textwidth}
  \centering
  \begin{tikzpicture}[scale=0.85]
\begin{axis}[
    title={Risk-neutral},
    ymin=0, ymax=450,
    minor y tick num = 1,
    area style,
    xlabel={Probability},
    ylabel={Frequency},
    xtick={0,0.2,0.4,0.6,0.8,1.0},
    ]
\addplot+[ybar interval,mark=no] plot coordinates { 
( 0.0 , 412 )
( 0.05 , 31 )
( 0.1 , 22 )
( 0.15 , 21 )
( 0.2 , 16 )
( 0.25 , 5 )
( 0.3 , 29 )
( 0.35 , 34 )
( 0.4 , 40 )
( 0.45 , 46 )
( 0.5 , 48 )
( 0.55 , 84 )
( 0.6 , 124 )
( 0.65 , 174 )
( 0.7 , 224 )
( 0.75 , 257 )
( 0.8 , 203 )
( 0.85 , 139 )
( 0.9 , 74 )
( 0.95 , 17 )
};
\end{axis}
\end{tikzpicture}
%
\end{subfigure}
\begin{subfigure}{.5\textwidth}
  \centering
\begin{tikzpicture}[scale=0.85]
\begin{axis}[
    title={Risk-averse},
    ymin=0, ymax=450,
    minor y tick num = 1,
    area style,
    xlabel={Probability},
    ylabel={Frequency},
    xtick={0,0.2,0.4,0.6,0.8,1.0},
    ]
\addplot+[ybar interval,mark=no] plot coordinates { 
( 0.0 , 0 )
( 0.05 , 0 )
( 0.1 , 0 )
( 0.15 , 2 )
( 0.2 , 10 )
( 0.25 , 44 )
( 0.3 , 107 )
( 0.35 , 190 )
( 0.4 , 268 )
( 0.45 , 322 )
( 0.5 , 304 )
( 0.55 , 243 )
( 0.6 , 200 )
( 0.65 , 150 )
( 0.7 , 95 )
( 0.75 , 37 )
( 0.8 , 19 )
( 0.85 , 6 )
( 0.9 , 3 )
( 0.95 , 0 )
};
\end{axis}
\end{tikzpicture}
%
\end{subfigure}
\end{figure}

Having motivated the risk-averse approach, let us formally define our risk measure CVaR: 
Let $\bm{\upsilon}\in\mathbb{R}^{|\mathcal{H}|}$ and $\alpha\in(0,1]$ such that $\alpha |\mathcal{H}| \in \mathbb{Z}_+$. Then, we define $s_\alpha ( \bm{\upsilon})$ as the average of the smallest $\alpha |\mathcal{H}|$ entries of $\bm{\upsilon}$. 
We are now ready to provide the risk-averse formulation of the problem:
\begin{equation} \label{eq:stochasticSampledAverse}
\max_{ \bm{x} \in \mathcal{X}}  s_\alpha \left( [ f_h (\bm{x})]_{h\in \mathcal{H}} \right) 
\end{equation}
Although $  s_\alpha(\cdot)$ is a nonlinear function, it has a  polyhedral representation as given in Proposition~\ref{prop:sConcavePoly}.
\begin{proposition}\label{prop:sConcavePoly}
\citep{nemirovskiLO} Consider the function $s_\alpha (\bm{\upsilon})$ defined above. Then, 
    \begin{equation*}
        s_\alpha (\bm{\upsilon}) =  \max_{\lambda, \mu} \left\{  \lambda-\frac{1}{\alpha |\mathcal{H}|}\sum_{h \in \mathcal{H}} \mu_h \ : \  \lambda - \mu_h \leq \upsilon_h, \mu_h \geq 0 \ h \in \mathcal{H}  \right\} .
    \end{equation*}    
\end{proposition}
Combining formulation~\eqref{eq:stochasticSampledDecompNeutralMILP} and Proposition~\ref{prop:sConcavePoly}, we provide an MILP formulation for the risk-averse version of the problem as below:
\begin{equation*} \label{eq:stochasticSampledDecompAverseMILP}
\textup{S-RA}_\alpha(\mathcal{H}):\quad
\max_{\bm{u}, \, \bm{v}, \, \bm{x} \in \mathcal{X}, \lambda, \mu } \    \left\{ \lambda - \frac{1}{\alpha |\mathcal{H}|} \sum_{h \in \mathcal{H}} \mu_h :   \lambda - \mu_h \leq \bm{q}^\top \bm{u_N^h} \ h  \in \mathcal{H} ,  \mu_h \geq 0 \ h  \in \mathcal{H}  ,  \eqref{eq:riskneutralMILP-consb} - \eqref{eq:riskneutralMILP-consf} \right\}. 
\end{equation*}
Notice that problem $\textup{S-RA}_1 (\mathcal{H})$ is equivalent to $\textup{S-RN}(\mathcal{H})$, therefore, the risk-neutral formulation is a special case of the risk-averse formulation. Therefore, in the remainder of the paper, we will mainly analyze the risk-averse formulation.

\subsection{Scenario Decomposition Algorithm}
\label{sec:scen-decomp-static}

It is challenging to solve  the formulations introduced above directly  using an off-the-shelf MILP solver even for moderate values of $|\mathcal{H}|$. Therefore, we resort to the scenario decomposition algorithm, originally proposed for risk-neutral problems in \citet{ahmed2013decomposition}. 
Our decomposition approach equipartitions the scenario set $\mathcal{H}$ into subsets $\{\mathcal{H}^p : p =1,\dots, P\}$ such that $\alpha |\mathcal{H}^p| \in \mathbb{Z}_+$, and solves  \textit{scenario batch subproblems}   as $\textup{S-RA}_\alpha(\mathcal{H}^p)$. 
The critical component of the scenario decomposition algorithm is to derive a lower bound (LB) and an upper bound (UB) from the scenario batch subproblems. The  result below provides these bounds for the risk-averse setting.


\begin{proposition}\label{prop:stochasticSampledDecompAverse}
Let $\alpha \in (0,1]$. Denote the optimal value of problem $\textup{S-RA}_\alpha(\mathcal{H})$ as $\hat z_\mathcal{H}$ and $\bm{x^{(p)}} = \argmax_{\bm{x} \in \mathcal{X}}  s_\alpha \left( [ f_h (\bm{x})]_{h\in \mathcal{H}^p} \right)$.
Then, 
\[
\underbrace{\max_{p  =1,\dots,  P} \left\{ s_\alpha \left( [ f_h (\bm{x^{(p)}})]_{h\in \mathcal{H}} \right) \right\}}_{LB}
\le 
\hat z_{\mathcal{H}} 
\le
\underbrace{\frac{1}{P} \sum_{p=1}^P  s_\alpha \left( [ f_h (\bm{x^{(p)}})]_{h\in \mathcal{H}^p} \right).}_{UB}
\]
\end{proposition}
\begin{proof}[Proof.]
We first prove the lower bound as follows:
Let $p  =1,\dots,  P$. Since the   solution $\bm{x^{(p)}}$ obtained from a batch subproblem $\textup{S-RA}_\alpha(\mathcal{H}^p)$ can be extended with appropriate  $\bm{u}$ and $\bm{v}$ values such that it is feasible for  $\textup{S-RA}_\alpha(\mathcal{H})$,  we deduce that $s_\alpha \left( [ f_h (\bm{x^{(p)}})]_{h\in \mathcal{H}} \right)$ is a lower bound on $\hat z_{\mathcal{H}}$.  Consequently, the best lower bound is obtained by taking the maximum over $p  =1,\dots,  P$.

We next prove the upper bound as follows: 
Let $\bm{x^*}$ be an optimal solution for $\textup{S-RA}_\alpha(\mathcal{H})$. Note that $\bm{x^*}$ is also feasible for   batch subproblem $\textup{S-RA}_\alpha(\mathcal{H}^p)$, $p  =1,\dots,  P$ when extended with appropriate  $\bm{u}$ and $\bm{v}$ values. Since $\bm{x^{(p)}}$ is an  optimal solution for $\textup{S-RA}_\alpha(\mathcal{H}^p)$, the relation \(s_\alpha \left( [ f_h (\bm{x^*})]_{h\in \mathcal{H}^p} \right) \leq s_\alpha \left( [ f_h (\bm{x^{(p)}})]_{h\in \mathcal{H}^p} \right)\)   holds true.
 The result follows since $\hat z_\mathcal{H} =  s_\alpha \left ( [ f_h (\bm{x^*})]_{h\in \mathcal{H}} \right) \le  \frac{1}{P} \sum_{p =1}^P  s_\alpha \left( [ f_h (\bm{x^*})]_{h\in \mathcal{H}^p} \right) $.
\hfill\Halmos
\end{proof}



We are now able to present our static scenario decomposition scheme in Algorithm \ref{alg:stochasticSampledDecompAverse}. 
We note that since our approach is based on SAA, upon termination with a solution $\bm{x^*}$   obtained from the ``in-sample optimization", we carry out an ``out-of-sample" evaluation using a fresh sample scenario set $\mathcal{H}'$. 
\begin{algorithm}[H]
\caption{Static Scenario Decomposition Algorithm.}
\label{alg:stochasticSampledDecompAverse}
\begin{algorithmic}
\STATE Set $t=1$, $\bm{x^*}=0$, $LB=0$ and $UB=1$.
\WHILE{$UB - LB > \epsilon$ and  $t \le \tau$ }
\STATE Solve $\textup{S-RA}_\alpha(\mathcal{H}^p)$ in parallel for   $p =1,\dots,  P$.  
\STATE If needed, update  $LB$ and $UB$  as in Proposition~\ref{prop:stochasticSampledDecompAverse} and the best feasible solution $\bm{x^*}$. 
\STATE Add the following no-good cut to the set~$\mathcal{X}$ for each distinct solution $\bm{\bm{x^{(p)}}}$ obtained:
\begin{equation}\label{eq:noGoodBase}
\sum_{ (n,k) : x_{n,k}^{(p)}=1} x_{n,k} \le |N|-1 .
\end{equation}
\STATE Increment $t=t+1$.
\ENDWHILE
\STATE Carry out an ``out-of-sample" analysis with a new sample $\mathcal{H}'$ by computing $s_\alpha \left( [ f_h (\bm{x^*})]_{h\in \mathcal{H}'}\right)$.
\end{algorithmic}
\end{algorithm}

\subsection{Enhancements}
\label{sec:improvements-static}

As the vanilla scenario decomposition algorithm (Algorithm \ref{alg:stochasticSampledDecompAverse}) is quite slow,   we introduce several enhancements to improve its computational efficiency in this section.


\subsubsection{Cartesian Cuts} \label{subsec:clusteringEnhancement}

By construction, each no-good cut~\eqref{eq:noGoodBase} eliminates exactly one solution in Algorithm~\ref{alg:stochasticSampledDecompAverse}. The main idea behind the Cartesian Cuts is to determine identify ``similar" solutions and eliminate them at the same time using a single inequality.

To be more precise, suppose that we solve scenario batch subproblems $\textup{S-RA}_\alpha(\mathcal{H}^p)$  each $p =1,\dots,  P$  and   obtain a set of solutions $\mathcal{T} := \{ \bm{x^{(p)}} : p  =1,\dots,  P\}$. 
Then,  using the K-means algorithm, we cluster these solutions into the clusters $\mathcal{T}_c$ for some $c\in\mathcal{C}$. For each cluster $c\in\mathcal{C}$, we obtain the sets   $\mathcal{A}_{c,n}= \{ k : x_{n,k} = 1 \text{ for some } \bm{x} \in \mathcal{T}_c \}$ and $\mathcal{K}_c = \mathcal{A}_{c,1} \times \dots \times \mathcal{A}_{c,N} $. Then, we add the following no-good cut:
\begin{equation} \label{eq:clusterNoGoodCut}
    \sum_{n=1}^N \sum_{k \in \mathcal{A}_{c,n}} x_{n,k} \leq N-1 \quad c\in\mathcal{C}.
\end{equation}
However, due to the Cartesian product operation, this inequality may cut solutions that do not come from the optimal solutions of scenario subproblems. Therefore, it is not valid unless the objective values of the solutions in the sets $\mathcal{K}_c \setminus \mathcal{T}_c $ are computed and LB is updated, if needed.

To summarize, the Cartesian Cuts~\eqref{eq:clusterNoGoodCut} have the advantage of eliminating more solutions per inequality (and, hence, per iteration) compared to the standard no-good cut~\eqref{eq:noGoodBase}. Therefore, they may have a positive effect on both LB and UB progress in Algorithm~\ref{alg:stochasticSampledDecompAverse}, and have the potential of reducing the number of iterations needed for convergence. On the other hand, the implementation of Cartesian Cuts has an overhead due to the computation of $s_\alpha \left( [ f_h (\bm{x})]_{h\in \mathcal{H}} \right)$ for every $\bm{x}\in\mathcal{K}_c \setminus \mathcal{T}_c$.

\subsubsection{Symmetry Breaking Inequalities and Symmetry-Enhanced Cartesian Cuts} \label{subsec:symmetryNoGood}

Interestingly, ``no in-take" action may be a part of an optimal solution for the Antibiotics Time Machine Problem (see, e.g. Figure~5a in~\citet{mesum2024}), meaning that we may opt for a shorter treatment plan than the one allowed with $N$.
This action can be modeled by adding a $d\times d$ identity matrix, denoted by $\bm{I}$, to the set of antibiotic matrices (we will refer to the index of this matrix by $I$). 
However, since the identity matrix is commutative with every matrix (e.g., applying antibiotics in the order  1-$I$-1 or 1-1-$I$ have the same effect), this results in symmetry in the problem, a known issue for the MILP solvers.   
In order to resolve this issue, we add the following symmetry breaking constraints:
\begin{equation}
\label{eq:sym-break}
    x_{1,I} = 0 \text{ and } x_{n,I} \leq x_{n+1,I} \quad  n = 1, \dots N-1.
\end{equation}

Once inequalities~\eqref{eq:sym-break} are considered, it is possible to further strengthen the Cartesian cuts~\eqref{eq:clusterNoGoodCut} introduced before as given in the following proposition:

\begin{proposition}\label{prop:symEnhancedCartCut}
Consider a cluster $\mathcal{T}_c$ of solutions as described before.
Let $N^*$ denote the first position that the identity matrix  is present  in $\mathcal{T}_c$, that is,
\[  N^* =    \begin{cases}
    n, & \text{ if }  I \notin \mathcal{A}_{c,n-1}, I \in \mathcal{A}_{c,n}  \\
    N/A,& \text{ otherwise.}
            \end{cases} \]
Let $N^I$ denote the first position for which the identity matrix  is the singleton in the sets $\mathcal{A}_{c,1},\dots, \mathcal{A}_{c,N}$, that is,
\[
    N^I =  \begin{cases}
    n, & \text{ if }   \{I\} \neq \mathcal{A}_{c,n-1}, \{I\} = \mathcal{A}_{c,n},   \\
    N/A,& \text{ otherwise.}
            \end{cases} 
\]
Consider the set $\mathcal{X}' = \{ \bm{x}\in\mathcal{X} : \eqref{eq:clusterNoGoodCut}, \eqref{eq:sym-break}\}$ and the following set definitions:
\begin{equation*}
\mathcal{\tilde{A}}^{\bar{N}}_{c,n} :=
\begin{cases}
     \mathcal{A}_{c,n}^{N} &\text{if } n \in \{1, N^* -1, \dots , \bar{N} -1\} \\
    \mathcal{A}_{c,n}^{N} \cup \{I\} &\text{if } n \in \{2,\dots,N^* -2\} \\
    \{I\} &\text{if } n = \bar{N}
\end{cases}
,
\mathcal{\tilde{A}}^{N}_{c,n} :=
\begin{cases}
     \mathcal{A}_{c,n} &\text{if } n \in \{1, N^* -1, \dots , N\} \\
    \mathcal{A}_{c,n}^{N} \cup \{I\} &\text{if } n \in \{2,\dots,N^* -2\}
\end{cases}.
\end{equation*}
Then, we have the following:
%
\begin{enumerate}[label=(\roman*)]
\item If $N^*$ is defined but $N^I$ is not defined, then the following inequalities are valid for $\mathcal{X}'$:
\begin{subequations}\label{eq:case1}
\begin{align}
\sum_{n=1}^{\bar{N}-1} \sum_{k \in \mathcal{\tilde{A}}_{c,n}^{\bar{N}}} x_{n,k} & \leq \bar{N}-1, \ \bar{N} = N^*,\dots,N-1 \label{eq:case1a}\\
\sum_{n=1}^{{N}} \sum_{k \in \mathcal{\tilde{A}}_{c,n}^{{N}}} x_{n,k} & \leq {N}-1 . \label{eq:case1b}
\end{align}
\end{subequations}



\item If both $N^*$ and $N^I$ are defined, then the following inequalities are valid for $\mathcal{X}'$:
\begin{equation}\label{eq:case2}
\sum_{n=1}^{\bar{N}-1} \sum_{k \in \mathcal{\tilde{A}}_{c,n}^{\bar{N}}} x_{n,k}  \leq \bar{N}-1, \ \bar{N} = N^*,\dots,N^{I}. 
\end{equation}
\end{enumerate}
\end{proposition}
\begin{proof}[Proof.]
We will first prove the validity of inequality~\eqref{eq:case1a} by showing that it is valid for both $\mathcal{X}'_0:=\mathcal{X}'\cap\{ x_{\bar{N},I} = 0 \}$ and $\mathcal{X}'_1:=\mathcal{X}'\cap\{ x_{\bar{N},I} = 1 \}$ (notice that we have $\mathcal{\tilde{A}}_{c,n}^{\bar{N}} = \{I\}$ in this case).  
\begin{itemize}
    \item Case 0: For the set $\mathcal{X}'_0$, the validity trivially follows since the remaining inequality 
    {$\sum_{n=1}^{\bar N - 1}\sum_{k \in \mathcal{\tilde{A}}_{c,n}^{\bar{N}} } x_{n,k} \leq \bar N - 1$ is simply redundant.}
\item Case 1: For the set $\mathcal{X}'_1$, consider an element $\gamma \in \mathcal{\tilde{A}}_{c,n}^{1} \times \dots \times \mathcal{\tilde{A}}_{c,n}^{\bar{N}} $.  We have two subcases:
\begin{itemize}
    \item Case 1a: We have $(\gamma_1,\dots,\gamma_{\bar{N}-1}) \in \Gamma:=\mathcal{A}_{c,1} \times \dots \times \mathcal{A}_{c,\Bar{N}-1}$.  {Since $(\gamma_1,\dots,\gamma_{\bar{N}-1},I,\dots,I) \in \mathcal{K}_c$,  inequality~\eqref{eq:case1a} is valid as it does not cut off any new solutions.}
    \item Case 1b: We have   $(\gamma_1,\dots,\gamma_{\bar{N}-1}) \not\in \Gamma$, which implies that there exists $n' \in \{2,\dots,N^*-2\}$ such that $\gamma_{n'}=I$. {Since $\gamma_{n'}=I$ and $I \not\in \mathcal{\tilde{A}}^{\bar{N}}_{c,N^*-1}$,  inequality~\eqref{eq:case1a} can only cut off solutions that violate the symmetry breaking constraint~\eqref{eq:sym-break}. Therefore, it remains valid.}
\end{itemize}
In both cases, inequality~\eqref{eq:case1a} does not cut off any  feasible solutions other than the ones coming from $\mathcal{K}_c$. Therefore, it is valid.
\end{itemize}
The validity of inequality~\eqref{eq:case1b} can be proven by repeating Case 1b of the proof of the validity of inequality~\eqref{eq:case1a}.
The validity of inequality~\eqref{eq:case2} can be proven by repeating the proof of the validity of inequality~\eqref{eq:case1a} by replacing $N-1$ with $N^I$.
\hfill \Halmos
\end{proof}
Note that inequalities \eqref{eq:case1}-\eqref{eq:case2} are stronger than the Cartesian cuts~\eqref{eq:clusterNoGoodCut}.

\subsubsection{Scenario Regrouping}\label{subsec:scenarioShuffling}

An interesting property of Algorithm~\ref{alg:stochasticSampledDecompAverse} is that the upper bound obtained via Proposition~\ref{prop:stochasticSampledDecompAverse} depends on the composition of the batches. In particular, even under the same solution, by regrouping batches, we may obtain different upper bounds. 
For example, suppose that we have eight scenarios and two batches with the same optimal solutions that give the values $ \ s_{0.25}( [{0.9}, 1.0, 1.0, 1.0] )=0.9$ and 
$s_{0.25}( [{0.5}, 0.5, 0.8, 0.8] )=0.5$, which lead to the upper bound
$ \text{UB}=\frac{0.9+0.5}{2} = 0.7$. Notice that these batches are \textit{unbalanced} as the scenarios with higher (resp. lower) probabilities are in Batch 1 (resp. 2). However, if we regroup the scenarios in a \textit{balanced} manner, we obtain the values
$ \ s_{0.25}( [{0.5}, 0.8, 0.9, 1.0] ) = s_{0.25}( [{0.5}, 0.8, 1.0, 1.0] )=0.5$, which lead to the upper bound 
$ \text{UB}=\frac{0.5+0.5}{2} = 0.5$.
%

Motivated by our observation that the optimal solution of one of the batch subproblems from the first iteration is generally optimal for the overall problem, we propose regrouping the scenarios in such a way that each batch is {balanced}, in the sense that, the scenarios in each batch share a similar distribution. In order to achieve this, we first compute that objective function value of each scenario $h$ and sort them. Let the sorted scenarios be labeled as $(h)$. Then, we regroup the scenarios into batches according to the following rule for $p=1,\dots,P$: $\mathcal{H}^p =\{ (h) :  (h) - 1 \mod P   \equiv p - 1  \}$.

\subsubsection{Warm Start}\label{subsec:cutinjection}

Notice that any feasible solution obtained for the instance with $N$ transitions can be extended to a feasible solution with the instance with $N+1$ transitions by simply appending an identity matrix for the last position. Motivated by this observation, we record the feasible solutions and the cuts generated by Algorithm~\ref{alg:stochasticSampledDecompAverse} for the instance with $N$ transitions, and \textit{warm start} the algorithm for the instance  with $N+1$ transitions by the lower bound obtained from the solutions and with the addition of the no-good cuts. This enhancement has the potential to help the progress on both lower and upper bounds, hence, decrease the number of iterations, at the cost of solving bulkier MILP models.


\section{Dynamic Stochastic Problem}
\label{sec:ABR-dynamic}

The static version of the problem studied in the previous section determines an antibiotic sequence at the beginning of the treatment. In the dynamic version studied in this section, it is allowed to observe the genotype of the bacteria after each antibiotic application. 
We  present how to formulate and solve the dynamic risk-averse version  in Sections~\ref{sec:ABR-dynamic-riskaverse}. Then, similar to the static version, we propose  a scenario decomposition algorithm in Section~\ref{sec:scen-decomp-dynamic} and several algorithmic enhancements to increase its computational efficiency in Section~\ref{sec:improvement-dynamic}.

\subsection{Risk-Averse Version}
\label{sec:ABR-dynamic-riskaverse}

The risk-averse dynamic  formulation of the Antibiotics Time Machine Problem can be stated similar to problem~\eqref{eq:stochasticSampledAverse} as 
\begin{equation*} \label{eq:stochasticSampledAverseDynamic}
\max_{ \bm{y} \in \mathcal{Y}}  s_\alpha \left( [ \ell_h (\bm{y})]_{h\in \mathcal{H}} \right) ,
\end{equation*}  
where the set
$ \mathcal{Y} := \{ \bm{y} \in \{0,1\}^{K \times d} : \sum_{k=1}^K y_{k,i}=1, i=1,\dots,d \}$ ensures that exactly one antibiotic~$k$ is chosen for each genotype~$i$ with the introduction of the binary variable $y_{k,i}$.
%
%
Here, $\ell_h (\bm{y})$ can be computed as $u_{N,i_w}^h$, where $i_w$ is the ID of the wild type, by fixing the $\bm{y}$ decisions and the initial condition $\bm{u_{n}^h} = \bm{r}$ in the following recursion ($\ell_h (\bm{y})$ can also be computed as the output of a forward DP, similar to the one used in~\citet{mesum2024}):
\begin{equation}
    {u^{h}_{n,j}} = \sum_{k=1}^K \sum_{i=1}^d u^{h}_{n-1,i} M_{ij}^{k,h} y_{k,i} \qquad n=1,\dots,N, j=1,\dots, d, h \in \mathcal{H}. \label{eq:dynamicrecursion}
\end{equation}


We now model the dynamic version of problem as an MILP.
Let $\bm{u_n}$ denote the probability distribution after the $n$-th antibiotic is applied as in the static version. 
To linearize the recursion~\eqref{eq:dynamicrecursion}, we define the decision variables $w_{n-1,k,i}^h$ and utilize the equation: $w_{n-1,k,i}^h = {{u_{n-1,i}^h}} y_{k,i}$.
We then provide an MILP formulation for the risk-averse
dynamic version of the problem as below:
\begin{subequations}\label{eq:dynamicRiskAverseMILP}
\begin{align}
 \textup{D-RA}_\alpha(\mathcal{H} ): \  &  
\max_{\bm{u}, \, \bm{w}, \, \bm{y}\in\mathcal{Y} , \lambda, \mu }   \lambda - \frac{1}{\alpha |\mathcal{H}|} \sum_{h \in \mathcal{H} } \mu_h \label{eq:dynamicRiskAverseObj} \\ 
\textup{s.t. } & \lambda - \mu_h \leq \bm{q}^\top \bm{u_N^h} &h& \in \mathcal{H}  \label{eq:dynamicRiskAverse-dual1} \\
\ & \mu_h \geq 0 &h& \in \mathcal{H} \label{eq:dynamicRiskAverse-dual2}\\
\ &  \bm{u_{0}^{h}} = \bm{r} &h&\in \mathcal{H} \label{eq:dynamicRiskAverse-recursionfirst}\\
 \  & u_{n,j}^{h} = \sum_{k=1}^K \sum_{i=1}^d   M_{ij}^{k,h} w_{n-1,k,i}^{h}  &n&=1,\dots,N,  j=1,\dots,d, h \in \mathcal{H}    \\
 \  &  w_{n-1,k,i}^{h} \ge 0, \  \  w_{n-1,k,i}^{h} \le y_{k,i}  &n&=1,\dots,N, k=1,\dots,K, i=1,\dots,d, h \in \mathcal{H}       \\ 
\  & \sum_{k = 1}^K w_{n-1,k,i}^{h}=  u_{n-1,i}^{h}   &n& =1,\dots,N, i=1,\dots,d, h \in \mathcal{H} \label{eq:dynamicRiskAverse-recursionlast} \\
  \  &  \sum_{j=1}^{d} u_{n,j}^{h} = 1 &n&=1,\dots,N, h \in \mathcal{H} \label{eq:dynamicRiskAverse-valid} \\ 
 \ &   \bm{u_{n}^{h}} \in \mathbb{R}_+^d   &n&=1,\dots,N,  h \in \mathcal{H} . \label{eq:dynamicRiskAverse-domainRestrictions}
\end{align}
\end{subequations}
In this formulation, the objective function~\eqref{eq:dynamicRiskAverseObj} together with constraints~\eqref{eq:dynamicRiskAverse-dual1}-\eqref{eq:dynamicRiskAverse-dual2} model the risk-averse objective function.  Constraint~\eqref{eq:dynamicRiskAverse-recursionfirst} ensures that the initial state distribution is $\bm{r}$. 
Constraints~\eqref{eq:dynamicRiskAverse-recursionfirst}-\eqref{eq:dynamicRiskAverse-recursionlast} linearize the recursion \eqref{eq:dynamicrecursion}. Constraint~\eqref{eq:dynamicRiskAverse-valid} is a valid equality which enforces that $\bm{u_n^h}$ is a probability vector. Finally, constraint~\eqref{eq:dynamicRiskAverse-domainRestrictions} gives the variable domain restrictions.

Although the risk-neutral dynamic version can be modeled and solved as a Markov Decision Process as in \citet{mesum2024}, the DP proposed in that paper does not directly extend to the risk-averse objective we study in this paper. Therefore, we will resort to the MILP proposed above and propose a scenario batch decomposition algorithm to solve it efficiently.

\subsection{Scenario Decomposition Algorithm}
\label{sec:scen-decomp-dynamic}

As in the static version,  $\textup{D-RA}_\alpha(\mathcal{H} )$  is also challenging to solve. We once  again propose a scenario decomposition algorithm that relies on lower and upper bounds derived in the following proposition (we omit the proof of this proposition due to its similarity to that of Proposition~\ref{prop:stochasticSampledDecompAverse}):
\begin{proposition}\label{prop:stochasticSampledDecompAverseDynamic}
Let $\alpha \in (0,1]$. Denote the optimal value of problem  $\textup{D-RA}_\alpha(\mathcal{H})$ as $\hat z_\mathcal{H}$ and $\bm{y^{(p)}} = \argmax_{\bm{y} \in \mathcal{Y}}  s_\alpha \left( [ \ell_h (\bm{y})]_{h\in \mathcal{H}^p} \right)$.
Then, 
\[
\underbrace{\max_{p  =1,\dots,  P} \left\{ s_\alpha \left( [ \ell_h (\bm{y^{(p)}})]_{h\in \mathcal{H}} \right) \right\}}_{LB}
\le 
\hat z_{\mathcal{H}} 
\le
\underbrace{\frac{1}{P} \sum_{p=1}^P  s_\alpha \left( [ \ell_h (\bm{y^{(p)}})]_{h\in \mathcal{H}^p} \right).}_{UB}
\]
\end{proposition}

Our dynamic 
  scenario decomposition algorithm is given below:
\begin{algorithm}[H]
\caption{Dynamic Scenario Decomposition Algorithm.}
\label{alg:dynamicRiskAverseAlgo}
\begin{algorithmic}
\STATE Set $t=1$, $\bm{y^*}=0$, $LB=0$ and $UB=1$.
\WHILE{$UB - LB > \epsilon$ and $t \le \tau$ }
\STATE Solve $\textup{D-RA}_\alpha(\mathcal{H}^p )$  in parallel for each $p\in P$.  
\STATE  If needed, update  $LB$ and $UB$  as in Proposition~\ref{prop:stochasticSampledDecompAverseDynamic}  and the best feasible solution $\bm{y^*}$ 
\STATE Add the following no-good cut to  the set~$\mathcal{Y}$ for each distinct solution $\bm{\bm{y^{(p)}}}$ obtained:
\begin{equation}
    \sum_{ (k,i) : y_{k,i}^{(p)}=1} y_{k,i} \le d-1 .\label{eq:dynamicNoGoodCut}
\end{equation}
\STATE Increment $t=t+1$.
\ENDWHILE
\STATE  Carry out an ``out-of-sample" analysis with a new sample $\mathcal{H}'$ by computing $ s_\alpha \left( [ \ell_h (\bm{y^*})]_{h\in \mathcal{H}'}\right)$.
\end{algorithmic}
\end{algorithm}

\subsection{Enhancements}
\label{sec:improvement-dynamic}

For the dynamic version of the problem, two of the enhancements introduced for the static version are directly applicable: i) scenario regrouping (Section \ref{subsec:scenarioShuffling}), ii)  warm start (Section \ref{subsec:cutinjection}). 
{We would like to note that for the latter enhancement, unlike the static version, we  cut the solutions coming from smaller instances after evaluating their objective function values since these values may improve as $N$ is increased.} 
{Although the Cartesian cut enhancement (Section \ref{subsec:clusteringEnhancement}) is also applicable for the dynamic version in principle, our preliminary experiments have showed that the number of solutions in clusters increases rapidly (as $d > N$ in our computational setting), resulting in prohibitively large number evaluations, and, therefore, slowing down our algorithm.}

As an additional enhancement, we introduce the notion of
  \textit{irrelevant} genotypes, which we now explain. 
Given an initial genotype $\bm{i_I}\in\{0,1\}^a$ and the length of treatment $N$, we define the set of irrelevant genotypes as
$
\mathcal{I}_N(\bm{i_I}):=\{ \bm{g}\in\{0,1\}^a : \| \bm{i_I} - \bm{g} \|_1+ \|\bm{g}\|_1 > N\}
$. Notice that these are exactly the genotypes for which the number of transitions from $\bm{i_I}$ to $\bm{g}$ and then to $\bm{0}$ requires more than $N$ antibiotic applications. Since this happens with zero probability no matter what the policy is, we can safely set $y_{I,i}=1$ from the very beginning for the ID of such genotypes. 
To give an example, for the instance in  Figure~\ref{fig:abr-image}, if the initial genotype is 001 and $N=2$, then the genotype 101 is irrelevant.
The identification of irrelevant genotypes is crucial for the progress of the upper bound in Algorithm~\ref{alg:dynamicRiskAverseAlgo}, since, otherwise, the algorithm might stall. 

\section{Computational Results}
\label{sec:compresults}

In this section, we provide the results of our computational experiments. We use a 64-bit workstation with two Intel(R) Xeon(R) Gold 6248R CPU (3.00GHz) processors (256 GB  RAM) and the Python programming language. We utilize  Gurobi 11 to solve the MILPs with the default settings, except the absolute optimality gap parameter is set to 0.001 and the time limit is set to two hours. We set $|\mathcal{H}|=|\mathcal{H}'|=2000$, $|\mathcal{H}^p|=50$, $\alpha = 0.1$, $\epsilon=0.01$ and $\tau=5$.  
We use the experimental data from \citet{mira2017statistical} with $K=23$ antibiotics and $d=2^4 $ genotypes, and  sample the transition probability  matrices as explained in Appendix~\ref{app:matrixConstruction}. 
We also add the identity matrix as an additional antibiotic to model the ``no in-take" action.




\subsection{Static Version}\label{subsec:staticCompExp}

In this section, we provide the results of the computational experiments that we have conducted for the static version of the problem.

\subsubsection{Computational Effort}

In this part, we first analyze the effect of enhancements we introduce in Section~\ref{sec:improvements-static}. For this purpose, we run our algorithm under the following settings:
\begin{itemize}
    \item Setting~0: No enhancements
    \item Setting~1: All enhancements except Cartesian Cuts (Section~\ref{subsec:clusteringEnhancement})
    \item Setting~2: All enhancements except symmetry breaking inequalities and symmetry-enhanced Cartesian Cuts   (Section~\ref{subsec:symmetryNoGood})
    \item Setting~3: All enhancements except scenario regrouping (Section~\ref{subsec:scenarioShuffling})
    \item Setting~4: All enhancements except warm start (Section~\ref{subsec:cutinjection})
    \item Setting~5: All enhancements
\end{itemize}

In Table \ref{tab:Static-EnhancementEffect}, for two initial genotype $\bm{i_I}$ and treatment length $N$ combinations, we report the number of iterations needed for Algorithm~\ref{alg:stochasticSampledDecompAverse}  to terminate (in column ``\#") along with the CPU time in seconds (in column ``Time").

\begin{table}[H]
\centering
\small
\caption{Effect of enhancements for the risk-averse static version in terms of computational effort.} \label{tab:Static-EnhancementEffect}
\begin{tabular}{|c|c|cr|cr|cr|cr|cr|cr|}
\hline
& & \multicolumn{2}{c|}{\textbf{Setting~0}} & \multicolumn{2}{c|}{\textbf{Setting~1}} & \multicolumn{2}{c|}{\textbf{Setting~2}} & \multicolumn{2}{c|}{\textbf{Setting~3}
} & \multicolumn{2}{c|}{\textbf{Setting~4}} & \multicolumn{2}{c|}{\textbf{Setting~5}} \\ \hline
{$N$} & \textbf{$\bm{i_I}$}& \textbf{\#} & \textbf{Time} & \textbf{\#}  & \textbf{Time} & \textbf{\#}  & \textbf{Time} & \textbf{\#} & \textbf{Time} & \textbf{\#} & \textbf{Time} 
& \textbf{\#} & \textbf{Time} \\ \hline
\textbf{4} & \textbf{0001}& 4& 2441& 1& 593& 2& 1246& 1& 694& 2& 1153& 1& 697\\
\textbf{5} & \textbf{1111}& 5& 5063& 4& 3538& 5& 5482& 4& 3676& 3& 2754& 3&  2650\\ \hline
\end{tabular}

\end{table}

As can be seen from Table \ref{tab:Static-EnhancementEffect}, Setting~5 is significantly more efficient that Setting~0 for the first instance. In terms of individual enhancements, we observe that Setting~2 and Setting~4 require more iterations than Setting~5, indicating the importance of symmetry exploitation and warm start, respectively. We also note that Setting~1 is faster than Setting~5, suggesting that the overhead of implementing the Cartesian cuts exceeding their benefit in this case.
As for the second instance, we observe that the exclusion of any enhancement increases the solution time, suggesting that the enhancements might be even more influential for instances with larger values of $N$. Unlike the first instance, we see that Setting~1 and Setting~3 are also slower than Setting~5 for the second instance, indicating the importance of the Cartesian cuts and scenario regrouping. 
We note that, in addition to these two instances, we report the results of the computational experiments for each of the settings introduced above for each initial genotype with $N=4$ in Appendix \ref{app:detailedComputations} as well.

Next, we run the risk-averse model for $N=5,\dots,8$ using all the enhancements (i.e., Setting~5), and report the results in Table~\ref{tab:risk-averseN=5-8static}.
We report the number of iterations, or the difference between UB and LB (in italic) after the last iteration  if  this difference is larger than $\epsilon$, under the column ``\#/\%". 
Despite the computational benefits brought by the enhancements, it is still challenging to solve the instances with larger $N$ values. Therefore, we apply the following filtering rules to eliminate the antibiotics that do not appear in the optimal solutions of the instances with smaller~$N$ values:
\begin{itemize}
    \item Filter I: This filter is applied to instances with $N=5,6$ and it is \textit{genotype-independent}, that is, the same filtered set of antibotics is used for any initial genotype $\bm{i_I}$.
    In particular, we restrict the set of antibiotics to those that appear in  the optimal solutions of instances with  $N=4$ for any $\bm{i_I}$. This reduces the number of antibiotics from 23 to 13.
    \item Filter II: This filter is applied to instances with $N=7,8$ and it is \textit{genotype-dependent}, that is, a different filtered set of antibotics is used for each initial genotype $\bm{i_I}$.
     In particular, we restrict the set of antibiotics for an initial genotype $\bm{i_I}$ to those that appear in  the optimal solutions of instances with $N=4,5,6$   for the same initial genotype $\bm{i_I}$. This reduces the number of antibiotics to a range of one to six, depending on the initial genotype.   
\end{itemize}

\begin{table}[H]
\centering
\small
\caption{Computational effort to solve the risk-averse static version for $N=5,\dots,8$ with different filtering mechanisms (no iteration limit is enforced for $N=5$).}
\label{tab:risk-averseN=5-8static}
\begin{tabular}{|c|cr|cr|cr|cr|cr|}
\hline
 & \multicolumn{2}{c|}{\textbf{$N=5$}} & \multicolumn{2}{c|}{\textbf{$N=5$ - Filter I}} & \multicolumn{2}{c|}{\textbf{$N=6$ - Filter I}} & \multicolumn{2}{c|}{\textbf{$N=7$ - Filter II}} & \multicolumn{2}{c|}{\textbf{$N=8$ - Filter II}} \\ \hline
\textbf{$\bm{i_I}$} & \textbf{\#/\%} & \textbf{Time} & \textbf{\#/\%} & \textbf{Time} & \textbf{\#/\%} & \textbf{Time} & \textbf{\#/\%} & \textbf{Time} & \textbf{\#/\%} & \textbf{Time} \\ \hline
\textbf{1000} & 1 & 5360 & 1 & 1045 & 1 & 6211 & 1 & 4 & 1 & 4 \\
\textbf{0100} & 1 & 2077 & 1 & 442 & 1 & 2605 & 1 & 4 & 1 & 4 \\
\textbf{0010} & 1 & 1647 & 1 & 408 & 1 & 2143 & 1 & 103 & 1 & 239 \\
\textbf{0001} & 1 & 3228 & 1 & 615 & 1 & 3835 & 2 & 1065 & 2 & 4184 \\
\textbf{1100} & 1 & 3705 & 1 & 715 & 1 & 4189 & 1 & 9 & 1 & 30 \\
\textbf{1010} & 1 & 1985 & 1 & 387 & 1 & 2633 & 2 & 453 & 2 & 1168 \\
\textbf{1001} & 2 & 7586 & 2 & 1667 & 2 & 9523 & 2 & 5585 & \textit{0.15} & 36398 \\
\textbf{0110} & 1 & 2213 & 1 & 378 & 1 & 1730 & 1 & 49 & 1 & 76 \\
\textbf{0101} & 3 & 14816 & 1 & 1033 & 1 & 3856 & 1 & 1484 & 2 & 9796 \\
\textbf{0011} & 2 & 4296 & 2 & 940 & 2 & 5380 & 2 & 449 & 2 & 1345 \\
\textbf{1110} & 2 & 5689 & 2 & 1005 & 2 & 6582 & 2 & 607 & 2 & 1952 \\
\textbf{1101} & 3 & 14438 & 2 & 1794 & 4 & 21925 & 4 & 3410 & 5 & 14021 \\
\textbf{1011} & 1 & 2004 & 1 & 406 & 2 & 5312 & 2 & 584 & 2 & 1546 \\
\textbf{0111} & 2 & 4780 & 2 & 850 & 2 & 5572 & 2 & 999 & 3 & 6231 \\
\textbf{1111} & 9 & 47852 & 3 & 2650 & 3 & 16716 & 3 & 3588 & \textit{0.02} & 26115 \\ \hline
\textbf{Avg} & 2.1 & 8112 & 1.5 & 956 & 1.7 & 6548 & 1.8 & 1226 & 2.3 & 6874 \\ \hline
\end{tabular}
\end{table}

Although we introduce the filtering mechanisms to relieve the computational burden, there is also a biological reasoning: some of the antibiotics from the data set of \citet{mira2017statistical} have the same active ingredient, therefore, it is conceivable that only some of those would appear in the optimal solutions. A further justification of the filtering mechanism can be found in the first two columns of Table~\ref{tab:risk-averseN=5-8static}, where we compare the computational effort for $N=5$ without and with filtering. In fact, the solutions obtained under these two experiment turn out to be identical (recall that we have decided the filtering based on the optimal solutions of instances with  $N=4$), yet the CPU time is decreased from 8112 seconds to 956 seconds on average. We note that filtering mechanisms allow us to solve the corresponding instances within a reasonable time. The only two time-outs are observed when $N=8$ and the initial genotypes are 1001 and 1111.


\subsubsection{Comparison of Risk-Averse and Risk-Neutral Solutions}

In this section, we compare  risk-averse and risk-neutral solutions from different perspectives. Our results are visualized in Figure~\ref{fig:riskneutral vs riskaverse static}, where we report six values for each initial genotype and $N=3,\dots,8$:
\begin{itemize}
    \item RA-In-10\%: The optimal value of the ``in-sample" optimization problem under the risk-averse objective.
    \item RA-Out-Avg: The average value of all ``out-of-sample" scenarios when the risk-averse solution is applied.
    \item RA-Out-10\%: The average value of the worst 10\%``out-of-sample" scenarios when the risk-averse solution is applied.
    \item RN-In-Avg: The optimal value of the ``in-sample" optimization problem under the risk-neutral objective.
    \item RN-Out-Avg: The average value of all ``out-of-sample" scenarios when the risk-neutral solution is applied.
    \item RN-Out-10\%: The average value of the worst 10\%``out-of-sample" scenarios when the risk-neutral solution is applied.
\end{itemize}

\begin{figure}[ht]

\caption{Comparison of risk-averse and risk-neutral solutions for the static version.}
\label{fig:riskneutral vs riskaverse static}

\begin{subfigure}{.475\textwidth}
\begin{tikzpicture}[scale=0.7]
    \begin{axis}[
    title={\large{$N=3$}},
        width=12cm, 
        height=8cm, 
        ylabel={Probability},
        ymin=0, ymax=1,
        xtick=data,
        symbolic x coords={1000, 0100, 0010, 0001, 1100, 1010, 1001, 0110, 0101, 0011, 1110, 1101, 1011, 0111, 1111},
                xticklabels={\textbf{1000}, 0100, 0010, \textit{0001}, \textbf{1100}, \textit{1010}, \textit{1001}, 0110, \textit{0101}, \textbf{0011}, 1110, \textit{1101}, \textbf{1011}, \textit{0111}, 1111},
        ytick={0, 0.2, 0.4, 0.6, 0.8, 1.0},
        xticklabel style={rotate=90, anchor=east},
        legend style={at={(1.03,1.15)}, anchor=south,legend columns=6},
        grid=both,
        grid style = {dashed}
    ]

\addplot[
        only marks,
        mark= x,
        color=darkgreen,line width=3pt,
        mark size=7.5pt
    ] coordinates {
        (1000, 0.227)
        (0100, 0.826)
        (0010, 0.644)
        (0001, 0.159)
        (1100, 0.214)
        (1010, 0.214)
        (1001, 0.095)
        (0110, 0.651)
        (0101, 0.156)
        (0011, 0.23)
        (1110, 0.338)
        (1101, 0.098)
        (1011, 0.135)
        (0111, 0.277)        (1111, 10)

    };
    \addlegendentry{RA-In-10\%  }

\addplot[
	only marks,
	fill = red,
	mark=triangle,line width=1.5pt,
	color=red,
	mark size=6pt
	] coordinates {
        (1000, 0.483)
        (0100, 0.937)
        (0010, 0.838)
        (0001, 0.225)
        (1100, 0.458)
        (1010, 0.361)
        (1001, 0.239)
        (0110, 0.842)
        (0101, 0.222)
        (0011, 0.585)
        (1110, 0.682)
        (1101, 0.37)
        (1011, 0.369)
        (0111, 0.387)        (1111, 10)

    };
    \addlegendentry{RA-Out-Avg}

\addplot[
        only marks,
        mark=triangle*,
        color=red,
        mark size=6pt
    ] coordinates {
        (1000, 0.231)
        (0100, 0.825)
        (0010, 0.639)
        (0001, 0.16)
        (1100, 0.218)
        (1010, 0.216)
        (1001, 0.095)
        (0110, 0.646)
        (0101, 0.156)
        (0011, 0.223)
        (1110, 0.363)
        (1101, 0.09)
        (1011, 0.129)
        (0111, 0.272)        (1111, 10)

    };
    \addlegendentry{RA-Out-10\%}


\addplot[
        only marks,
        mark= + ,
        color=darkgreen,line width=3pt,
        mark size = 7.5pt
    ]
	 coordinates {
        (1000, 0.641)
        (0100, 0.937)
        (0010, 0.842)
        (0001, 0.548)
        (1100, 0.61)
        (1010, 0.581)
        (1001, 0.571)
        (0110, 0.847)
        (0101, 0.501)
        (0011, 0.593)
        (1110, 0.683)
        (1101, 0.51)
        (1011, 0.424)
        (0111, 0.643)
                 (1111, 10)

    };
    \addlegendentry{RN-In-Avg}

\addplot[
        only marks,
        fill = blue,
        mark=o,line width=1.5pt,
        color=blue,
        mark size = 3.5pt
    ] coordinates {
        (1000, 0.627)
        (0100, 0.937)
        (0010, 0.838)
        (0001, 0.536)
        (1100, 0.595)
        (1010, 0.595)
        (1001, 0.56)
        (0110, 0.842)
        (0101, 0.501)
        (0011, 0.586)
        (1110, 0.682)
        (1101, 0.533)
        (1011, 0.428)
        (0111, 0.637)
        (1111, 10)

    };
\addlegendentry{RN-Out-Avg}

\addplot[
        only marks,
        mark=*,
        color=blue,
        fill = blue,
        mark size = 3.5pt
    ] coordinates {
        (1000, 0)  
        (0100, 0.825)
        (0010, 0.639)
        (0001, 0.08)
        (1100, 0)
        (1010, 0)
        (1001, 0)
        (0110, 0.646)
        (0101, 0)
        (0011, 0.192)
        (1110, 0.363)
        (1101, 0)
        (1011, 0)
        (0111, 0.282)
        (1111, 10)
    };
\addlegendentry{RN-Out-10\%}

 
\addplot[
        only marks,
        mark= x,
        color=white,line width=3pt,
        mark size=7.5pt
    ] coordinates {
        (1000, 10.227)
        (0100, 10.826)
        (0010, 10.644)
        (0001, 10.159)
        (1100, 10.214)
        (1010, 10.214)
        (1001, 10.095)
        (0110, 10.651)
        (0101, 10.156)
        (0011, 10.23)
        (1110, 10.338)
        (1101, 10.098)
        (1011, 10.135)
        (0111, 10.277)
        (1111, 0.227)
        };


    \end{axis}
\end{tikzpicture}

\end{subfigure}
\begin{subfigure}{.5\textwidth}
\begin{tikzpicture}[scale=0.7]
    \begin{axis}[
    title={\large{$N=4$}},
        width=12cm, 
        height=8cm, 
        ymin=0, ymax=1,
        xtick=data,
        symbolic x coords={1000, 0100, 0010, 0001, 1100, 1010, 1001, 0110, 0101, 0011, 1110, 1101, 1011, 0111, 1111},
                xticklabels={\textbf{1000}, 0100, 0010, \textit{0001}, \textbf{1100}, 1010, \textit{1001}, 0110, \textit{0101}, 0011, 1110, \textit{1101}, \textbf{1011}, \textbf{0111}, \textit{1111}},
        ytick={0, 0.2, 0.4, 0.6, 0.8, 1.0},
        xticklabel style={rotate=90, anchor=east},
        legend style={at={(10.25,1.1)}, anchor=south,legend columns=-1},
        grid=both,
        grid style = {dashed}
    ]

    \addplot[ 
        only marks,
        mark= x,
        color=white,line width=2pt,
        mark size=7.5pt
    ] coordinates {
        (1000, 10.227)
        (0100, 10.826)
        (0010, 10.645)
        (0001, 10.159)
        (1100, 10.214)
        (1010, 10.422)
        (1001, 10.159)
        (0110, 10.651)
        (0101, 10.156)
        (0011, 10.515)
        (1110, 10.355)
        (1101, 10.098)
        (1011, 10.135)
        (0111, 10.294)
        (1111, 10.099)
    };
    \addlegendentry{ }
    
    \addplot[
        only marks,
        mark= x,
        color=darkgreen,line width=3pt,
        mark size=7.5pt
    ] coordinates {
        (1000, 0.227)
        (0100, 0.826)
        (0010, 0.645)
        (0001, 0.159)
        (1100, 0.214)
        (1010, 0.422)
        (1001, 0.159)
        (0110, 0.651)
        (0101, 0.156)
        (0011, 0.515)
        (1110, 0.355)
        (1101, 0.098)
        (1011, 0.135)
        (0111, 0.294)
        (1111, 0.099)
    };

    \addplot[
        only marks,
        fill = red,
        mark=triangle,line width=1.5pt,
        color=red,
        mark size=6pt
    ] coordinates {
        (1000, 0.483)
        (0100, 0.937)
        (0010, 0.839)
        (0001, 0.225)
        (1100, 0.458)
        (1010, 0.66)
        (1001, 0.351)
        (0110, 0.842)
        (0101, 0.222)
        (0011, 0.75)
        (1110, 0.667)
        (1101, 0.37)
        (1011, 0.369)
        (0111, 0.635)
        (1111, 0.367)
    };

    \addplot[
        only marks,
        mark=triangle*,
        color=red,
        mark size=6pt
    ] coordinates {
        (1000, 0.231)
        (0100, 0.825)
        (0010, 0.64)
        (0001, 0.16)
        (1100, 0.218)
        (1010, 0.418)
        (1001, 0.16)
        (0110, 0.646)
        (0101, 0.156)
        (0011, 0.514)
        (1110, 0.369)
        (1101, 0.09)
        (1011, 0.129)
        (0111, 0.306)
        (1111, 0.092)
    };

        \addplot[
        only marks,
        mark= + ,
        color=darkgreen,line width=3pt,
        mark size = 7.5pt
    ] coordinates {
        (1000, 0.641)
        (0100, 0.937)
        (0010, 0.843)
        (0001, 0.548)
        (1100, 0.61)
        (1010, 0.665)
        (1001, 0.571)
        (0110, 0.847)
        (0101, 0.541)
        (0011, 0.757)
        (1110, 0.683)
        (1101, 0.57)
        (1011, 0.459)
        (0111, 0.643)
        (1111, 0.506)
    };

    \addplot[
        only marks,
        fill = blue,
        mark=o,line width=1.5pt,
        color=blue,
        mark size = 3.5pt
    ] coordinates {
        (1000, 0.627)
        (0100, 0.937)
        (0010, 0.839)
        (0001, 0.536)
        (1100, 0.595)
        (1010, 0.66)
        (1001, 0.56)
        (0110, 0.842)
        (0101, 0.529)
        (0011, 0.75)
        (1110, 0.682)
        (1101, 0.56)
        (1011, 0.448)
        (0111, 0.637)
        (1111, 0.529)
    };

    \addplot[
        only marks,
        mark=*,
        color=blue,
        fill = blue,
        mark size = 3.5pt
    ] coordinates {
        (1000, 0)
        (0100, 0.825)
        (0010, 0.64)
        (0001, 0.08)
        (1100, 0)
        (1010, 0.418)
        (1001, 0)
        (0110, 0.646)
        (0101, 0.081)
        (0011, 0.514)
        (1110, 0.363)
        (1101, 0)
        (1011, 0)
        (0111, 0.282)
        (1111, 0)
    };

    \end{axis}
\end{tikzpicture}
\end{subfigure}

\begin{subfigure}{.475\textwidth}
\begin{tikzpicture}[scale=0.7]
    \begin{axis}[
    title={\large{$N=5$}},
        width=12cm, 
        height=8cm, 
        ylabel={Probability},
        ymin=0, ymax=1,
        xtick=data,
        symbolic x coords={1000, 0100, 0010, 0001, 1100, 1010, 1001, 0110, 0101, 0011, 1110, 1101, 1011, 0111, 1111},
                xticklabels={\textbf{1000}, 0100, 0010, \textbf{0001}, \textbf{1100}, 1010, \textit{1001}, 0110, \textit{0101}, 0011, \textbf{1110}, \textit{1101}, \textbf{1011}, 0111, \textit{1111}},
        ytick={0, 0.2, 0.4, 0.6, 0.8, 1.0},
        xticklabel style={rotate=90, anchor=east},
        grid=both,
        grid style = {dashed}
    ]
\addplot[
        only marks,
        mark= x,
        color=darkgreen,line width=3pt,
        mark size=7.5pt
    ] coordinates {
        (1000, 0.227)
        (0100, 0.826)
        (0010, 0.645)
        (0001, 0.317)
        (1100, 0.214)
        (1010, 0.422)
        (1001, 0.159)
        (0110, 0.651)
        (0101, 0.156)
        (0011, 0.515)
        (1110, 0.414)
        (1101, 0.159)
        (1011, 0.289)
        (0111, 0.507)
        (1111, 0.099)
    };

\addplot[
	only marks,
	fill = red,
	mark=triangle,line width=1.5pt,
	color=red,
	mark size=6pt
	] coordinates {
        (1000, 0.483)
        (0100, 0.937)
        (0010, 0.839)
        (0001, 0.517)
        (1100, 0.458)
        (1010, 0.66)
        (1001, 0.351)
        (0110, 0.842)
        (0101, 0.222)
        (0011, 0.75)
        (1110, 0.7)
        (1101, 0.351)
        (1011, 0.47)
        (0111, 0.754)
        (1111, 0.367)
        
    };

\addplot[
        only marks,
        mark=triangle*,
        color=red,
        mark size=6pt
    ] coordinates {
        (1000, 0.231)
        (0100, 0.825)
        (0010, 0.64)
        (0001, 0.317)
        (1100, 0.218)
        (1010, 0.418)
        (1001, 0.16)
        (0110, 0.646)
        (0101, 0.156)
        (0011, 0.514)
        (1110, 0.427)
        (1101, 0.16)
        (1011, 0.286)
        (0111, 0.516)
        (1111, 0.092)
        
    };


\addplot[
        only marks,
        mark= + ,
        color=darkgreen,line width=3pt,
        mark size = 7.5pt
    ] coordinates {
        (1000, 0.641)
        (0100, 0.937)
        (0010, 0.843)
        (0001, 0.548)
        (1100, 0.61)
        (1010, 0.665)
        (1001, 0.571)
        (0110, 0.847)
        (0101, 0.541)
        (0011, 0.757)
        (1110, 0.706)
        (1101, 0.57)
        (1011, 0.506)
        (0111, 0.758)
        (1111, 0.566)
        
    };

\addplot[
        only marks,
        fill = blue,
        mark=o,line width=1.5pt,
        color=blue,
        mark size = 3.5pt
    ] coordinates {
        (1000, 0.627)
        (0100, 0.937)
        (0010, 0.839)
        (0001, 0.545)
        (1100, 0.595)
        (1010, 0.66)
        (1001, 0.56)
        (0110, 0.842)
        (0101, 0.529)
        (0011, 0.75)
        (1110, 0.706)
        (1101, 0.56)
        (1011, 0.529)
        (0111, 0.754)
        (1111, 0.556)
    };

\addplot[
        only marks,
        mark=*,
        color=blue,
        fill = blue,
        mark size = 3.5pt
    ] coordinates {
        (1000, 0)
        (0100, 0.825)
        (0010, 0.64)
        (0001, 0)
        (1100, 0)
        (1010, 0.418)
        (1001, 0)
        (0110, 0.646)
        (0101, 0.081)
        (0011, 0.514)
        (1110, 0.393)
        (1101, 0)
        (1011, 0)
        (0111, 0.516)
        (1111, 0)
    };


    \end{axis}
\end{tikzpicture}

\end{subfigure}
\begin{subfigure}{.5\textwidth}
\begin{tikzpicture}[scale=0.7]
    \begin{axis}[
    title={\large{$N=6$}},
        width=12cm, 
        height=8cm, 
        ymin=0, ymax=1,
        xtick=data,
        symbolic x coords={1000, 0100, 0010, 0001, 1100, 1010, 1001, 0110, 0101, 0011, 1110, 1101, 1011, 0111, 1111},
                xticklabels={\textbf{1000}, 0100, 0010, \textbf{0001}, \textbf{1100}, 1010, \textbf{1001}, 0110, \textbf{0101}, 0011, \textbf{1110}, \textit{1101}, \textit{1011}, 0111, \textit{1111}},
        ytick={0, 0.2, 0.4, 0.6, 0.8, 1.0},
        xticklabel style={rotate=90, anchor=east},
        grid=both,
        grid style = {dashed}
    ]
    
    \addplot[
        only marks,
        mark= x,
        color=darkgreen,line width=3pt,
        mark size=7.5pt
    ] coordinates {
        (1000, 0.227)
        (0100, 0.826)
        (0010, 0.645)
        (0001, 0.317)
        (1100, 0.214)
        (1010, 0.485)
        (1001, 0.192)
        (0110, 0.651)
        (0101, 0.318)
        (0011, 0.565)
        (1110, 0.414)
        (1101, 0.159)
        (1011, 0.289)
        (0111, 0.507)
        (1111, 0.159)
        
    };

    \addplot[
        only marks,
        fill = red,
        mark=triangle,line width=1.5pt,
        color=red,
        mark size=6pt
    ] coordinates {
        (1000, 0.483)
        (0100, 0.937)
        (0010, 0.839)
        (0001, 0.517)
        (1100, 0.458)
        (1010, 0.717)
        (1001, 0.453)
        (0110, 0.842)
        (0101, 0.517)
        (0011, 0.774)
        (1110, 0.7)
        (1101, 0.351)
        (1011, 0.47)
        (0111, 0.754)
        (1111, 0.35)
        
    };

    \addplot[
        only marks,
        mark=triangle*,
        color=red,
        mark size=6pt
    ] coordinates {
        (1000, 0.231)
        (0100, 0.825)
        (0010, 0.64)
        (0001, 0.317)
        (1100, 0.218)
        (1010, 0.484)
        (1001, 0.186)
        (0110, 0.646)
        (0101, 0.319)
        (0011, 0.562)
        (1110, 0.427)
        (1101, 0.16)
        (1011, 0.286)
        (0111, 0.516)
        (1111, 0.16)
        
    };

        \addplot[
        only marks,
        mark= + ,
        color=darkgreen,line width=3pt,
        mark size = 7.5pt
    ] coordinates {
        (1000, 0.641)
        (0100, 0.937)
        (0010, 0.843)
        (0001, 0.6)
        (1100, 0.61)
        (1010, 0.721)
        (1001, 0.571)
        (0110, 0.847)
        (0101, 0.543)
        (0011, 0.779)
        (1110, 0.706)
        (1101, 0.57)
        (1011, 0.59)
        (0111, 0.758)
        (1111, 0.566)

    };

    \addplot[
        only marks,
        fill = blue,
        mark=o,line width=1.5pt,
        color=blue,
        mark size = 3.5pt
    ] coordinates {
        (1000, 0.627)
        (0100, 0.937)
        (0010, 0.839)
        (0001, 0.593)
        (1100, 0.595)
        (1010, 0.717)
        (1001, 0.56)
        (0110, 0.842)
        (0101, 0.54)
        (0011, 0.774)
        (1110, 0.706)
        (1101, 0.56)
        (1011, 0.592)
        (0111, 0.754)
        (1111, 0.556)
        
    };

    \addplot[
        only marks,
        mark=*,
        color=blue,
        fill = blue,
        mark size = 3.5pt
    ] coordinates {
        (1000, 0)
        (0100, 0.825)
        (0010, 0.64)
        (0001, 0.058)
        (1100, 0)
        (1010, 0.484)
        (1001, 0)
        (0110, 0.646)
        (0101, 0)
        (0011, 0.562)
        (1110, 0.393)
        (1101, 0)
        (1011, 0.261)
        (0111, 0.516)
        (1111, 0)
        
    };
    \end{axis}
\end{tikzpicture}
\end{subfigure}

\begin{subfigure}{.475\textwidth}
\begin{tikzpicture}[scale=0.7]
    \begin{axis}[
    title={\large{$N=7$}},
        width=12cm, 
        height=8cm, 
         xlabel={Initial Genotype},
        ylabel={Probability},
        ymin=0, ymax=1,
        xtick=data,
        symbolic x coords={1000, 0100, 0010, 0001, 1100, 1010, 1001, 0110, 0101, 0011, 1110, 1101, 1011, 0111, 1111},
                xticklabels={\textbf{1000}, 0100, 0010, \textbf{0001}, \textbf{1100}, 1010, \textbf{1001}, 0110, \textbf{0101}, 0011, \textbf{1110}, \textit{1101}, \textit{1011}, 0111, \textit{1111}},
        ytick={0, 0.2, 0.4, 0.6, 0.8, 1.0},
        xticklabel style={rotate=90, anchor=east},
        grid=both,
        grid style = {dashed}
    ]
\addplot[
        only marks,
        mark= x,
        color=darkgreen,line width=3pt,
        mark size=7.5pt
    ] coordinates {
(1000, 0.227)
(0100, 0.826)
(0010, 0.645)
(0001, 0.351)
(1100, 0.214)
(1010, 0.485)
(1001, 0.232)
(0110, 0.651)
(0101, 0.318)
(0011, 0.565)
(1110, 0.439)
(1101, 0.159)
(1011, 0.292)
(0111, 0.513)
(1111, 0.159)

    };

\addplot[
	only marks,
	fill = red,
	mark=triangle,line width=1.5pt,
	color=red,
	mark size=6pt
	] coordinates {
(1000, 0.483)
(0100, 0.937)
(0010, 0.839)
(0001, 0.566)
(1100, 0.458)
(1010, 0.717)
(1001, 0.452)
(0110, 0.842)
(0101, 0.517)
(0011, 0.774)
(1110, 0.714)
(1101, 0.351)
(1011, 0.516)
(0111, 0.761)
(1111, 0.35)

    };

\addplot[
        only marks,
        mark=triangle*,
        color=red,
        mark size=6pt
    ] coordinates {
(1000, 0.231)
(0100, 0.825)
(0010, 0.64)
(0001, 0.353)
(1100, 0.218)
(1010, 0.484)
(1001, 0.222)
(0110, 0.646)
(0101, 0.319)
(0011, 0.562)
(1110, 0.447)
(1101, 0.16)
(1011, 0.296)
(0111, 0.522)
(1111, 0.16)

    };


\addplot[
        only marks,
        mark= + ,
        color=darkgreen,line width=3pt,
        mark size = 7.5pt
    ] coordinates {
(1000, 0.641)
(0100, 0.937)
(0010, 0.843)
(0001, 0.616)
(1100, 0.61)
(1010, 0.721)
(1001, 0.571)
(0110, 0.847)
(0101, 0.594)
(0011, 0.779)
(1110, 0.708)
(1101, 0.57)
(1011, 0.608)
(0111, 0.764)
(1111, 0.566)

    };

\addplot[
        only marks,
        fill = blue,
        mark=o,line width=1.5pt,
        color=blue,
        mark size = 3.5pt
    ] coordinates {
(1000, 0.627)
(0100, 0.937)
(0010, 0.839)
(0001, 0.607)
(1100, 0.595)
(1010, 0.717)
(1001, 0.56)
(0110, 0.842)
(0101, 0.587)
(0011, 0.774)
(1110, 0.709)
(1101, 0.56)
(1011, 0.604)
(0111, 0.761)
(1111, 0.556)

    };

\addplot[
        only marks,
        mark=*,
        color=blue,
        fill = blue,
        mark size = 3.5pt
    ] coordinates {
(1000, 0)
(0100, 0.825)
(0010, 0.64)
(0001, 0.055)
(1100, 0)
(1010, 0.484)
(1001, 0)
(0110, 0.646)
(0101, 0.061)
(0011, 0.562)
(1110, 0.392)
(1101, 0)
(1011, 0.222)
(0111, 0.522)
(1111, 0)

    };


    \end{axis}
\end{tikzpicture}

\end{subfigure}
\begin{subfigure}{.5\textwidth}
\begin{tikzpicture}[scale=0.7]
    \begin{axis}[
    title={\large{$N=8$}},
        width=12cm, 
        height=8cm, 
         xlabel={Initial Genotype},
        ymin=0, ymax=1,
        xtick=data,
        symbolic x coords={1000, 0100, 0010, 0001, 1100, 1010, 1001, 0110, 0101, 0011, 1110, 1101, 1011, 0111, 1111},
                xticklabels={\textbf{1000}, 0100, 0010, \textbf{0001}, \textbf{1100}, 1010, \textbf{1001}, 0110, \textbf{0101}, 0011, \textbf{1110}, \textit{1101}, \textit{1011}, \textbf{0111}, \textit{1111}},
        ytick={0, 0.2, 0.4, 0.6, 0.8, 1.0},
        xticklabel style={rotate=90, anchor=east},
        grid=both,
        grid style = {dashed}
    ]
    
    \addplot[
        only marks,
        mark= x,
        color=darkgreen,line width=3pt,
        mark size=7.5pt
    ] coordinates {
(1000, 0.227)
(0100, 0.826)
(0010, 0.645)
(0001, 0.351)
(1100, 0.214)
(1010, 0.495)
(1001, 0.254)
(0110, 0.651)
(0101, 0.351)
(0011, 0.568)
(1110, 0.439)
(1101, 0.159)
(1011, 0.301)
(0111, 0.513)
(1111, 0.161)

    };

    \addplot[
        only marks,
        fill = red,
        mark=triangle,line width=1.5pt,
        color=red,
        mark size=6pt
    ] coordinates {
(1000, 0.483)
(0100, 0.937)
(0010, 0.839)
(0001, 0.562)
(1100, 0.458)
(1010, 0.727)
(1001, 0.519)
(0110, 0.842)
(0101, 0.564)
(0011, 0.779)
(1110, 0.714)
(1101, 0.351)
(1011, 0.489)
(0111, 0.761)
(1111, 0.344)

    };

    \addplot[
        only marks,
        mark=triangle*,
        color=red,
        mark size=6pt
    ] coordinates {
(1000, 0.231)
(0100, 0.825)
(0010, 0.64)
(0001, 0.352)
(1100, 0.218)
(1010, 0.498)
(1001, 0.239)
(0110, 0.646)
(0101, 0.353)
(0011, 0.57)
(1110, 0.447)
(1101, 0.16)
(1011, 0.297)
(0111, 0.522)
(1111, 0.159)

    };

        \addplot[
        only marks,
        mark= + ,
        color=darkgreen,line width=3pt,
        mark size = 7.5pt
    ] coordinates {
(1000, 0.641)
(0100, 0.937)
(0010, 0.843)
(0001, 0.616)
(1100, 0.61)
(1010, 0.73)
(1001, 0.571)
(0110, 0.847)
(0101, 0.61)
(0011, 0.784)
(1110, 0.712)
(1101, 0.57)
(1011, 0.639)
(0111, 0.764)
(1111, 0.566)

    };

    \addplot[
        only marks,
        fill = blue,
        mark=o,line width=1.5pt,
        color=blue,
        mark size = 3.5pt
    ] coordinates {
(1000, 0.627)
(0100, 0.937)
(0010, 0.839)
(0001, 0.607)
(1100, 0.595)
(1010, 0.727)
(1001, 0.56)
(0110, 0.842)
(0101, 0.602)
(0011, 0.779)
(1110, 0.711)
(1101, 0.56)
(1011, 0.637)
(0111, 0.762)
(1111, 0.556)

    };

    \addplot[
        only marks,
        mark=*,
        color=blue,
        fill = blue,
        mark size = 3.5pt
    ] coordinates {
(1000, 0)
(0100, 0.825)
(0010, 0.64)
(0001, 0.055)
(1100, 0)
(1010, 0.498)
(1001, 0)
(0110, 0.646)
(0101, 0.057)
(0011, 0.57)
(1110, 0.35)
(1101, 0)
(1011, 0.295)
(0111, 0.476)
(1111, 0)

    };
    \end{axis}
\end{tikzpicture}
\end{subfigure}

\end{figure}
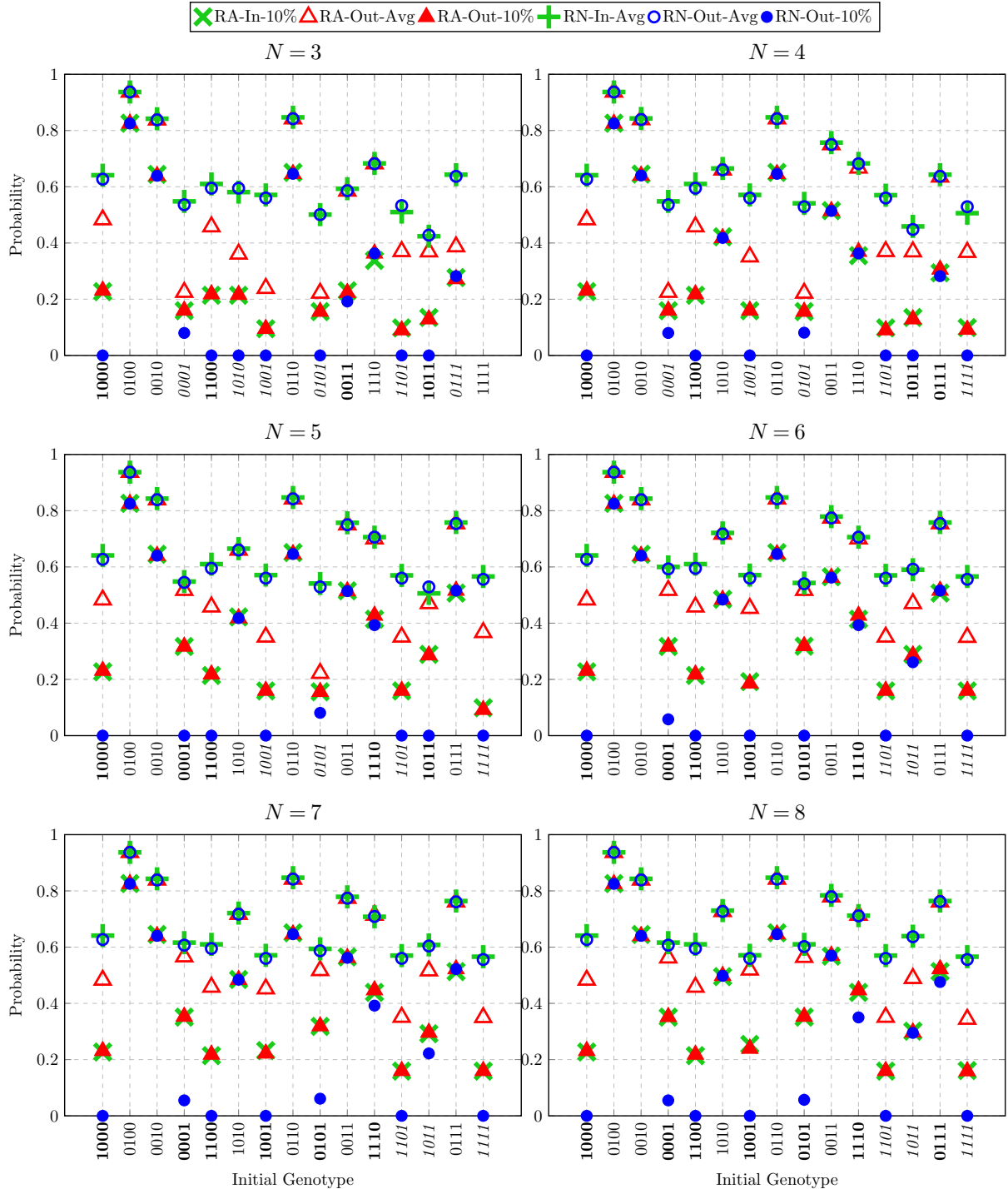

We have several observations from our analysis: Firstly, our   approach is quite reliable as the in-sample and out-of-sample performances under both risk-averse and risk-neutral objectives are close to each other. In fact, RA-In-10\% values, symbolized with a cross sign, and RA-Out-10\% values, symbolized with a filled triangle, almost always coincide (similarly, RN-In-Avg values, symbolized with a plus sign, and RN-Out-Avg values, symbolized with an empty circle, coincide as well). Therefore,  our results suggest that the choice of the scenario size as $|\mathcal{H}|=2000$ is reasonable.

Secondly, the results illustrate an expected trade-off between risk-averse and risk-neutral objectives. Clearly, the risk-averse solution has a better RA-Out-10\% value compared to the RN-Out-10\% value obtained from the risk-neutral solution. However, this comes with a sacrifice from the average performance as RN-Out-Avg values are typically larger than  RA-Out-Avg values. Depending on how this trade-off plays out, we roughly divide the genotypes into three groups:
\begin{itemize}
    \item Group 1: These are \textit{indifferent} genotypes for which risk-averse and risk-neutral solutions are indistinguishable (e.g., genotypes 0100, 0010, 0110, 1110 for $N=3$).
    \item Group 2: These are \textit{good} genotypes for which the increase in the worst 10\% performance is higher than the decrease in the average performance when switched from the risk-neutral solution to the risk-averse solution  (e.g., genotypes 1000, 1100, 0011, 1011 for $N=3$). The labels of these genotypes are  boldface.
    \item Group 3: These are \textit{bad} genotypes for which the increase in the worst 10\% performance is lower than the decrease in the average performance when switched from the risk-neutral solution to the risk-averse solution  (e.g., genotypes 0001, 1010, 1001, 0101, 1101, 0111 for $N=3$). The labels of these genotypes are italicized.
\end{itemize}












We notice   that the number of \textit{bad} (resp. \textit{good}) genotypes  decreases (resp. increases)  consistently as $N$ increases.
%
For $N=4$, the number of \textit{good} and \textit{bad} genotypes are four and five, respectively. 
Compared to the risk-neutral approach, solving the problem in a risk-averse manner improves the worst 10\% performance by 7.3\%, computed as the difference of RA-Out-10\% and RN-Out-10\%, on average whereas it decreases the average performance by 10.3\%, computed as the difference of RA-Out-Avg and RN-Out-Avg.
Going up to $N=6$, the number of \textit{good} and \textit{bad} genotypes become six and three, respectively. Compared to the risk-neutral solution, the risk-averse solution improves the worst 10\% performance by 10.6\% on average while decreasing the average performance by 6.8\%. 
Finally, solving the problem for $N=8$, we see that the number of \textit{good} genotypes increases to seven while the number of \textit{bad} genotypes is still three. The risk-averse solution is better than the risk-neutral solution by 11.6\% in terms of the worst 10\% performance, which comes with a 6.5\% sacrifice from the average performance.

As illustrated above, switching from the risk-neutral solution to the risk-averse solution comes with a significant gain from the worst 10\% performance albeit a considerable sacrifice from the average performance in most cases even for \textit{good} genotypes. 
These observations illustrate the price of risk-aversion in the static version. 
However, as $N$ increases, the aforementioned gain increases and the sacrifice decreases on average.

We remark that we do not solve the problem for larger $N$ values since it would require too much computational effort whereas the potential benefit is projected to be small (notice that the objective functions values seem to stabilize for the larger $N$ values 
reported in Figure~\ref{fig:riskneutral vs riskaverse static}).

\subsection{Dynamic Version}\label{subsec:dynamicComp}

In this section, we present the results for the dynamic version of the problem. We note that as this approach allows for intermediate observations, we can take more informed steps, and, therefore, the success of the dynamic version will be higher compared to the success of  the static version in terms of reversing the antibiotic resistance.

\subsubsection{Computational Effort}

 
We run Algorithm~\ref{alg:dynamicRiskAverseAlgo} with the enhancements introduced in Section~\ref{sec:improvement-dynamic} for $N=3,\dots,8$, and report the results in Table~\ref{tab:risk-averseN=3-8dynamic}, whose format is similar to that of Table~\ref{tab:risk-averseN=5-8static}.

\begin{table}[H]
\centering
\small
\caption{Computational effort to solve the risk-averse dynamic version for $N=3,\dots,8$ with different filtering mechanisms.}
\label{tab:risk-averseN=3-8dynamic}
\begin{tabular}{|c|cr|cr|cr|cr|cr|cr|}
\hline
  & \multicolumn{2}{c|}{\textbf{$N=3$}} & \multicolumn{2}{c|}{\textbf{$N=4$}} & \multicolumn{2}{c|}{\textbf{$N=5$ - F}} & \multicolumn{2}{c|}{\textbf{$N=6$ - F}} & \multicolumn{2}{c|}{\textbf{$N=7$ - F}} & \multicolumn{2}{c|}{\textbf{$N=8$ - F}} \\ \hline
\textbf{$\bm{i_I}$} & \multicolumn{1}{r}{\textbf{\#/\%}} & \textbf{Time} & \multicolumn{1}{r}{\textbf{\#/\%}} & \textbf{Time} & \multicolumn{1}{r}{\textbf{\#/\%}} & \textbf{Time} & \multicolumn{1}{r}{\textbf{\#/\%}} & \textbf{Time} & \multicolumn{1}{r}{\textbf{\#/\%}} & \textbf{Time} & \multicolumn{1}{r}{\textbf{\#/\%}} & \textbf{Time} \\ \hline
\textbf{1000} & \textit{0.02} & 466 & \textit{0.02} & 613 & \textit{0.01}& 640
& \textit{0.02} & 819 & \textit{0.03} & 1122 & \textit{0.03} & 1181 \\
\textbf{0100} & \textit{0.02} & 923 & \textit{0.02} & 1459 & 1& 121
& 1 & 158 & 1 & 211 & 1 & 236 \\
\textbf{0010} & 2 & 184 & 2 & 231 & 1& 120
& 1 & 148 & 1 & 180 & 1 & 204 \\
\textbf{0001} & \textit{0.03} & 914 & \textit{0.03} & 1368 & \textit{0.01}& 656
& 4 & 668 & \textit{0.01} & 1115 & \textit{0.01} & 1161 \\
\textbf{1100} & 3 & 241 & \textit{0.04} & 5380 & \textit{0.03}& 642
& \textit{0.04} & 979 & \textit{0.04} & 1216 & \textit{0.04} & 1442 \\
\textbf{1010} & 1 & 83 & \textit{0.03} & 9369 & 1& 125
& \textit{0.02} & 1032 & \textit{0.02} & 1285 & \textit{0.02} & 1564 \\
\textbf{1001} & 2 & 173 & \textit{0.03} & 18593 & \textit{0.03}& 645
& \textit{0.03} & 932 & \textit{0.03} & 1161 & \textit{0.03} & 1333 \\
\textbf{0110} & 1 & 80 & \textit{0.01} & 1999 & 1& 119
& 1 & 162 & 1 & 190 & 1 & 234 \\
\textbf{0101} & 2 & 171 & \textit{0.05} & 15262 & \textit{0.02}& 634
& \textit{0.03} & 1008 & \textit{0.03} & 1232 & 2 & 547 \\
\textbf{0011} & 2 & 165 & \textit{0.03} & 1877 & 1& 123
& \textit{0.02} & 848 & \textit{0.02} & 1018 & \textit{0.02} & 1103 \\
\textbf{1110} & \textit{0.03} & 812 & \textit{0.03} & 1406 & \textit{0.01}& 724
& \textit{0.02} & 1114 & \textit{0.05} & 1569 & \textit{0.05} & 1845 \\
\textbf{1101} & \textit{0.05} & 840 & \textit{0.05} & 1690 & \textit{0.03}& 668
& \textit{0.06} & 993 & \textit{0.05} & 1323 & \textit{0.05} & 1439 \\
\textbf{1011} & \textit{0.03} & 821 & \textit{0.03} & 1233 & \textit{0.02}& 728
& \textit{0.02} & 904 & \textit{0.02} & 1246 & \textit{0.02} & 1287 \\
\textbf{0111} & 4 & 589 & 4 & 812 & 1& 166
& \textit{0.03} & 937 & \textit{0.04} & 1342 & \textit{0.04} & 1436 \\
\textbf{1111} & - & - & \textit{0.06} & 53206 & \textit{0.02}& 787
& \textit{0.07} & 1105 & \textit{0.07} & 1381 & \textit{0.06} & 1600 \\ \hline
\textbf{Avg} & 3.4 & 461 & 4.7 & 7633 & 3.4& 459& 4.1 & 787 & 4.2 & 1039 & 4.0 & 1107 \\ \hline
\end{tabular}
\end{table}

The dynamic version turns out to be more challenging that the static version due to slower progress on the upper bound, therefore, we frequently reach the iteration limit of $\tau$ even for small values of $N=3,4$. In addition, solving the subproblems becomes time consuming for larger values of $N$ as well,  therefore, we introduce another filtering mechanism:  We first solve the risk-averse dynamic version for $N=4$ and record the solutions as $ \bm{y}^{[\bm{i_I}]} $   when the initial genotype is $\bm{i_I}$.  
Then, if an antibiotic $k$ is used when the stochastic process is at genotype $i$ for \textit{any} initial genotype, we limit the set of available antibiotics for genotype $i$ to be $K_i :=   \{k = 1,\dots,K: \ y_{k,i}^{[\bm{i_I}]} = 1  \text{ for some }  \bm{i_I}\}$ except the genotype $1111$, for which we do not apply any filter. 
This filtering rule limits the antibiotics to be used and decreases the solution time of the MILP significantly. 
The results reported in Table~\ref{tab:risk-averseN=3-8dynamic}
show that although Algorithm~\ref{alg:dynamicRiskAverseAlgo} reaches the iteration limit in most cases, the gap obtained upon termination is generally small, ranging between 0.01-0.07. 

\subsubsection{Comparison of Risk-Averse and Risk-Neutral Solutions} 

Similar to Figure~\ref{fig:riskneutral vs riskaverse static}, we now compare   risk-averse and risk-neutral solutions for the dynamic version in   Figure~\ref{fig:riskneutral vs riskaverse dynamic},  for each initial genotype and $N=3,\dots,8$. A significant difference between  the results of static and dynamic versions is that both average and worst 10\% performances are much higher when intermediate observations are allowed.

As before, the results suggest that 
  our   approach is again  reliable as the in-sample and out-of-sample performances under both risk-averse and risk-neutral objectives are close to each other. 
  Next, we will discuss the trade-off between risk-averse and risk-neutral objectives. Using the same terminology, we observe that the number of \textit{good}  genotypes  increases and the number of \textit{bad}  genotypes  decreases  rapidly as $N$ increases. 








For   $N=4$,  we observe  that  nine genotypes are \textit{good} and only one genotype is \textit{bad}. Compared to the risk-neutral solution, the risk-averse solution improves the worst 10\% performance by 14.0\%  while losing 5.9\% from the average performance. Moving up to $N=6$, the number of \textit{good} genotypes increases to ten while the number of \textit{bad} genotypes stays the same. The risk-averse solution performs 14.2\% better in the worst 10\% performance with  only  3.1\% sacrifice from the average performance. Finally, at $N=8$, we have  ten \textit{good} and zero \textit{bad} genotypes. The risk-averse solution has a 17.2\% better worst 10\% performance compared to the risk-neutral solution with only 1.8\% loss from the average performance. Therefore, we can say that as $N$ increases, the benefit of solving the problem in our proposed risk-averse approach increases considerably.




To summarize,  when switched from the risk-neutral solution to the risk-averse solution, we observe significantly better worst 10\% performance with relatively small reductions from the average performance. 
These observations suggest that the price of risk-aversion is quite low in the dynamic version, and our proposed approach is arguably more compatible with this version.

\begin{figure}[h]
\caption{Comparison of risk-averse and risk-neutral solutions for the dynamic version.}
\label{fig:riskneutral vs riskaverse dynamic}

\begin{subfigure}{.475\textwidth}
\begin{tikzpicture}[scale=0.7]
    \begin{axis}[
    title={\large{$N=3$}},
        width=12cm, 
        height=8cm, 
        ylabel={Probability},
        ymin=0, ymax=1,
        xtick=data,
        symbolic x coords={1000, 0100, 0010, 0001, 1100, 1010, 1001, 0110, 0101, 0011, 1110, 1101, 1011, 0111, 1111},
                xticklabels={\textbf{1000}, 0100, 0010, \textbf{0001}, \textbf{1100}, \textbf{1010}, \textit{1001}, 0110, \textit{0101}, 0011, \textbf{1110}, \textit{1101}, \textbf{1011}, 0111, 1111},
        ytick={0, 0.2, 0.4, 0.6, 0.8, 1.0},
        xticklabel style={rotate=90, anchor=east},
legend style={at={(1.03,1.15)}, anchor=south,legend columns=6},        grid=both,
        grid style = {dashed}
    ]

\addplot[
        only marks,
        mark= x,
        color=darkgreen,line width=3pt,
        mark size=7.5pt
    ] coordinates {
        (1000, 0.405)
        (0100, 0.901)
        (0010, 0.768)
        (0001, 0.412)
        (1100, 0.224)
        (1010, 0.339)
        (1001, 0.129)
        (0110, 0.725)
        (0101, 0.158)
        (0011, 0.233)
        (1110, 0.49)
        (1101, 0.146)
        (1011, 0.253)
        (0111, 0.53)

                    (1111, 10)

    };
    \addlegendentry{RA-In-10\%  }

\addplot[
	only marks,
	fill = red,
	mark=triangle,line width=1.5pt,
	color=red,
	mark size=6pt
	] coordinates {
        (1000, 0.688)
        (0100, 0.967)
        (0010, 0.939)
        (0001, 0.635)
        (1100, 0.399)
        (1010, 0.562)
        (1001, 0.268)
        (0110, 0.876)
        (0101, 0.224)
        (0011, 0.586)
        (1110, 0.737)
        (1101, 0.285)
        (1011, 0.51)
        (0111, 0.773)

               (1111, 10)

    };
    \addlegendentry{RA-Out-Avg}

\addplot[
        only marks,
        mark=triangle*,
        color=red,
        mark size=6pt
    ] coordinates {
(1000, 0.404)
(0100, 0.901)
(0010, 0.787)
(0001, 0.425)
(1100, 0.234)
(1010, 0.346)
(1001, 0.127)
(0110, 0.724)
(0101, 0.159)
(0011, 0.227)
(1110, 0.521)
(1101, 0.149)
(1011, 0.263)
(0111, 0.533)
               (1111, 10)

    };
    \addlegendentry{RA-Out-10\%}


\addplot[
        only marks,
        mark= + ,
        color=darkgreen,line width=3pt,
        mark size = 7.5pt
    ]
	 coordinates {
    (1000, 0.704)
    (0100, 0.971)
    (0010, 0.937)
    (0001, 0.714)
    (1100, 0.582)
    (1010, 0.628)
    (1001, 0.537)
    (0110, 0.877)
    (0101, 0.504)
    (0011, 0.593)
    (1110, 0.754)
    (1101, 0.536)
    (1011, 0.565)
    (0111, 0.777)
                 (1111, 10)

    };
    \addlegendentry{RN-In-Avg}

\addplot[
        only marks,
        fill = blue,
        mark=o,line width=1.5pt,
        color=blue,
        mark size = 3.5pt
    ] coordinates {
(1000, 0.721)
(0100, 0.972)
(0010, 0.939)
(0001, 0.718)
(1100, 0.595)
(1010, 0.632)
(1001, 0.551)
(0110, 0.876)
(0101, 0.508)
(0011, 0.586)
(1110, 0.759)
(1101, 0.551)
(1011, 0.569)
(0111, 0.773)

        (1111, 10)

    };
\addlegendentry{RN-Out-Avg}

\addplot[
        only marks,
        mark=*,
        color=blue,
        fill = blue,
        mark size = 3.5pt
    ] coordinates {
(1000, 0)
(0100, 0.893)
(0010, 0.788)
(0001, 0.199)
(1100, 0)
(1010, 0.167)
(1001, 0)
(0110, 0.724)
(0101, 0.044)
(0011, 0.227)
(1110, 0.44)
(1101, 0)
(1011, 0.162)
(0111, 0.533)

        (1111, 10)
    };
\addlegendentry{RN-Out-10\%}

 
\addplot[
        only marks,
        mark= x,
        color=white,line width=3pt,
        mark size=7.5pt
    ] coordinates {
        (1000, 10)
        (0100, 10.893)
        (0010, 10.788)
        (0001, 10.199)
        (1100, 10)
        (1010, 10.167)
        (1001, 10)
        (0110, 10.724)
        (0101, 10.044)
        (0011, 10.227)
        (1110, 10.44)
        (1101, 10)
        (1011, 10.162)
        (0111, 10.533)
        (1111, 0.3)
        };

    \end{axis}
\end{tikzpicture}

\end{subfigure}
\begin{subfigure}{.5\textwidth}
\begin{tikzpicture}[scale=0.7]
    \begin{axis}[
    title={\large{$N=4$}},
        width=12cm, 
        height=8cm, 
        ymin=0, ymax=1,
        xtick=data,
        symbolic x coords={1000, 0100, 0010, 0001, 1100, 1010, 1001, 0110, 0101, 0011, 1110, 1101, 1011, 0111, 1111},
                xticklabels={\textbf{1000}, 0100, 0010, \textbf{0001}, \textbf{1100}, \textbf{1010}, \textbf{1001}, 0110, \textbf{0101}, 0011, \textbf{1110}, \textit{1101}, \textbf{1011}, 0111, \textbf{1111}},
        ytick={0, 0.2, 0.4, 0.6, 0.8, 1.0},
        xticklabel style={rotate=90, anchor=east},
legend style={at={(10.25,1.1)}, anchor=south,legend columns=-1},        grid=both,
        grid style = {dashed}
    ]

    \addplot[ 
        only marks,
        mark= x,
        color=white,line width=2pt,
        mark size=7.5pt
    ] coordinates {
        (1000, 10.227)
        (0100, 10.826)
        (0010, 10.645)
        (0001, 10.159)
        (1100, 10.214)
        (1010, 10.422)
        (1001, 10.159)
        (0110, 10.651)
        (0101, 10.156)
        (0011, 10.515)
        (1110, 10.355)
        (1101, 10.098)
        (1011, 10.135)
        (0111, 10.294)
        (1111, 10.099)
    };
    \addlegendentry{ }
    
    \addplot[
        only marks,
        mark= x,
        color=darkgreen,line width=3pt,
        mark size=7.5pt
    ] coordinates {
        (1000, 0.405)
        (0100, 0.901)
        (0010, 0.768)
        (0001, 0.412)
        (1100, 0.421)
        (1010, 0.654)
        (1001, 0.332)
        (0110, 0.881)
        (0101, 0.432)
        (0011, 0.682)
        (1110, 0.49)
        (1101, 0.146)
        (1011, 0.253)
        (0111, 0.53)
        (1111, 0.313)
        
    };

    \addplot[
        only marks,
        fill = red,
        mark=triangle,line width=1.5pt,
        color=red,
        mark size=6pt
    ] coordinates {
        (1000, 0.688)
        (0100, 0.967)
        (0010, 0.939)
        (0001, 0.635)
        (1100, 0.643)
        (1010, 0.843)
        (1001, 0.541)
        (0110, 0.962)
        (0101, 0.649)
        (0011, 0.891)
        (1110, 0.737)
        (1101, 0.285)
        (1011, 0.51)
        (0111, 0.773)
        (1111, 0.548)

    };

    \addplot[
        only marks,
        mark=triangle*,
        color=red,
        mark size=6pt
    ] coordinates {
        (1000, 0.404)
        (0100, 0.901)
        (0010, 0.787)
        (0001, 0.425)
        (1100, 0.433)
        (1010, 0.669)
        (1001, 0.341)
        (0110, 0.888)
        (0101, 0.429)
        (0011, 0.693)
        (1110, 0.521)
        (1101, 0.149)
        (1011, 0.263)
        (0111, 0.533)
        (1111, 0.326)

    };

        \addplot[
        only marks,
        mark= + ,
        color=darkgreen,line width=3pt,
        mark size = 7.5pt
    ] coordinates {
        (1000, 0.704)
        (0100, 0.971)
        (0010, 0.937)
        (0001, 0.714)
        (1100, 0.712)
        (1010, 0.851)
        (1001, 0.645)
        (0110, 0.962)
        (0101, 0.742)
        (0011, 0.897)
        (1110, 0.754)
        (1101, 0.537)
        (1011, 0.565)
        (0111, 0.777)
        (1111, 0.631)

    };

    \addplot[
        only marks,
        fill = blue,
        mark=o,line width=1.5pt,
        color=blue,
        mark size = 3.5pt
    ] coordinates {
        (1000, 0.721)
        (0100, 0.972)
        (0010, 0.939)
        (0001, 0.718)
        (1100, 0.728)
        (1010, 0.854)
        (1001, 0.652)
        (0110, 0.963)
        (0101, 0.752)
        (0011, 0.897)
        (1110, 0.759)
        (1101, 0.551)
        (1011, 0.569)
        (0111, 0.773)
        (1111, 0.639)

    };

    \addplot[
        only marks,
        mark=*,
        color=blue,
        fill = blue,
        mark size = 3.5pt
    ] coordinates {
        (1000, 0)
        (0100, 0.893)
        (0010, 0.788)
        (0001, 0.199)
        (1100, 0.151)
        (1010, 0.606)
        (1001, 0.043)
        (0110, 0.883)
        (0101, 0.187)
        (0011, 0.687)
        (1110, 0.44)
        (1101, 0)
        (1011, 0.162)
        (0111, 0.533)
        (1111, 0.07)

    };

    \end{axis}
\end{tikzpicture}
\end{subfigure}

\begin{subfigure}{.475\textwidth}
\begin{tikzpicture}[scale=0.7]
    \begin{axis}[
    title={\large{$N=5$}},
        width=12cm, 
        height=8cm, 
        ylabel={Probability},
        ymin=0, ymax=1,
        xtick=data,
        symbolic x coords={1000, 0100, 0010, 0001, 1100, 1010, 1001, 0110, 0101, 0011, 1110, 1101, 1011, 0111, 1111},
                xticklabels={\textbf{1000}, 0100, 0010, \textbf{0001}, \textbf{1100}, \textbf{1010}, \textbf{1001}, 0110, \textbf{0101}, 0011, {1110}, \textbf{1101}, \textbf{1011}, \textit{0111}, \textbf{1111}},
        ytick={0, 0.2, 0.4, 0.6, 0.8, 1.0},
        xticklabel style={rotate=90, anchor=east},
        grid=both,
        grid style = {dashed}
    ]
\addplot[
        only marks,
        mark= x,
        color=darkgreen,line width=3pt,
        mark size=7.5pt
    ] coordinates {
        (1000, 0.544)
        (0100, 0.936)
        (0010, 0.882)
        (0001, 0.641)
        (1100, 0.421)
        (1010, 0.656)
        (1001, 0.332)
        (0110, 0.881)
        (0101, 0.433)
        (0011, 0.682)
        (1110, 0.674)
        (1101, 0.302)
        (1011, 0.563)
        (0111, 0.748)
        (1111, 0.316)

    };

\addplot[
	only marks,
	fill = red,
	mark=triangle,line width=1.5pt,
	color=red,
	mark size=6pt
	] coordinates {
        (1000, 0.755)
        (0100, 0.98)
        (0010, 0.979)
        (0001, 0.794)
        (1100, 0.643)
        (1010, 0.846)
        (1001, 0.541)
        (0110, 0.962)
        (0101, 0.645)
        (0011, 0.892)
        (1110, 0.882)
        (1101, 0.522)
        (1011, 0.743)
        (0111, 0.908)
        (1111, 0.558)

    };

\addplot[
        only marks,
        mark=triangle*,
        color=red,
        mark size=6pt
    ] coordinates {
        (1000, 0.543)
        (0100, 0.935)
        (0010, 0.901)
        (0001, 0.651)
        (1100, 0.433)
        (1010, 0.671)
        (1001, 0.341)
        (0110, 0.888)
        (0101, 0.437)
        (0011, 0.696)
        (1110, 0.714)
        (1101, 0.31)
        (1011, 0.571)
        (0111, 0.754)
        (1111, 0.332)

    };


\addplot[
        only marks,
        mark= + ,
        color=darkgreen,line width=3pt,
        mark size = 7.5pt
    ] coordinates {
        (1000, 0.8)
        (0100, 0.983)
        (0010, 0.978)
        (0001, 0.83)
        (1100, 0.713)
        (1010, 0.851)
        (1001, 0.645)
        (0110, 0.962)
        (0101, 0.742)
        (0011, 0.897)
        (1110, 0.879)
        (1101, 0.645)
        (1011, 0.784)
        (0111, 0.923)
        (1111, 0.632)

    };

\addplot[
        only marks,
        fill = blue,
        mark=o,line width=1.5pt,
        color=blue,
        mark size = 3.5pt
    ] coordinates {
        (1000, 0.807)
        (0100, 0.984)
        (0010, 0.98)
        (0001, 0.839)
        (1100, 0.728)
        (1010, 0.854)
        (1001, 0.652)
        (0110, 0.963)
        (0101, 0.752)
        (0011, 0.897)
        (1110, 0.883)
        (1101, 0.652)
        (1011, 0.794)
        (0111, 0.924)
        (1111, 0.639)

    };

\addplot[
        only marks,
        mark=*,
        color=blue,
        fill = blue,
        mark size = 3.5pt
    ] coordinates {
        (1000, 0.45)
        (0100, 0.931)
        (0010, 0.899)
        (0001, 0.299)
        (1100, 0.152)
        (1010, 0.606)
        (1001, 0.043)
        (0110, 0.879)
        (0101, 0.187)
        (0011, 0.687)
        (1110, 0.7)
        (1101, 0.043)
        (1011, 0.418)
        (0111, 0.749)
        (1111, 0.07)
    };


    \end{axis}
\end{tikzpicture}

\end{subfigure}
\begin{subfigure}{.5\textwidth}
\begin{tikzpicture}[scale=0.7]
    \begin{axis}[
    title={\large{$N=6$}},
        width=12cm, 
        height=8cm, 
        ymin=0, ymax=1,
        xtick=data,
        symbolic x coords={1000, 0100, 0010, 0001, 1100, 1010, 1001, 0110, 0101, 0011, 1110, 1101, 1011, 0111, 1111},
                xticklabels={\textbf{1000}, 0100, 0010, \textbf{0001}, \textbf{1100}, \textbf{1010}, \textbf{1001}, 0110, \textbf{0101}, \textbf{0011}, \textbf{1110}, {1101}, \textbf{1011}, \textit{0111}, \textbf{1111}},
        ytick={0, 0.2, 0.4, 0.6, 0.8, 1.0},
        xticklabel style={rotate=90, anchor=east},
        grid=both,
        grid style = {dashed}
    ]
    
    \addplot[
        only marks,
        mark= x,
        color=darkgreen,line width=3pt,
        mark size=7.5pt
    ] coordinates {
        (1000, 0.544)
        (0100, 0.936)
        (0010, 0.882)
        (0001, 0.641)
        (1100, 0.531)
        (1010, 0.791)
        (1001, 0.483)
        (0110, 0.924)
        (0101, 0.61)
        (0011, 0.816)
        (1110, 0.674)
        (1101, 0.302)
        (1011, 0.563)
        (0111, 0.748)
        (1111, 0.465)

    };

    \addplot[
        only marks,
        fill = red,
        mark=triangle,line width=1.5pt,
        color=red,
        mark size=6pt
    ] coordinates {
        (1000, 0.755)
        (0100, 0.98)
        (0010, 0.979)
        (0001, 0.794)
        (1100, 0.754)
        (1010, 0.931)
        (1001, 0.697)
        (0110, 0.98)
        (0101, 0.776)
        (0011, 0.956)
        (1110, 0.882)
        (1101, 0.522)
        (1011, 0.743)
        (0111, 0.908)
        (1111, 0.752)

    };

    \addplot[
        only marks,
        mark=triangle*,
        color=red,
        mark size=6pt
    ] coordinates {
        (1000, 0.543)
        (0100, 0.935)
        (0010, 0.901)
        (0001, 0.651)
        (1100, 0.535)
        (1010, 0.812)
        (1001, 0.481)
        (0110, 0.933)
        (0101, 0.617)
        (0011, 0.835)
        (1110, 0.714)
        (1101, 0.31)
        (1011, 0.571)
        (0111, 0.754)
        (1111, 0.493)
        
    };

        \addplot[
        only marks,
        mark= + ,
        color=darkgreen,line width=3pt,
        mark size = 7.5pt
    ] coordinates {
        (1000, 0.8)
        (0100, 0.983)
        (0010, 0.978)
        (0001, 0.83)
        (1100, 0.797)
        (1010, 0.928)
        (1001, 0.756)
        (0110, 0.982)
        (0101, 0.807)
        (0011, 0.956)
        (1110, 0.879)
        (1101, 0.645)
        (1011, 0.784)
        (0111, 0.923)
        (1111, 0.765)

    };

    \addplot[
        only marks,
        fill = blue,
        mark=o,line width=1.5pt,
        color=blue,
        mark size = 3.5pt
    ] coordinates {
        (1000, 0.807)
        (0100, 0.984)
        (0010, 0.98)
        (0001, 0.839)
        (1100, 0.803)
        (1010, 0.93)
        (1001, 0.762)
        (0110, 0.983)
        (0101, 0.822)
        (0011, 0.958)
        (1110, 0.883)
        (1101, 0.652)
        (1011, 0.794)
        (0111, 0.924)
        (1111, 0.763)

    };

    \addplot[
        only marks,
        mark=*,
        color=blue,
        fill = blue,
        mark size = 3.5pt
    ] coordinates {
        (1000, 0.45)
        (0100, 0.931)
        (0010, 0.899)
        (0001, 0.299)
        (1100, 0.45)
        (1010, 0.783)
        (1001, 0.052)
        (0110, 0.922)
        (0101, 0.24)
        (0011, 0.817)
        (1110, 0.7)
        (1101, 0.043)
        (1011, 0.418)
        (0111, 0.749)
        (1111, 0.201)

    };
    \end{axis}
\end{tikzpicture}
\end{subfigure}

\begin{subfigure}{.475\textwidth}
\begin{tikzpicture}[scale=0.7]
    \begin{axis}[
    title={\large{$N=7$}},
        width=12cm, 
        height=8cm, 
         xlabel={Initial Genotype},
        ylabel={Probability},
        ymin=0, ymax=1,
        xtick=data,
        symbolic x coords={1000, 0100, 0010, 0001, 1100, 1010, 1001, 0110, 0101, 0011, 1110, 1101, 1011, 0111, 1111},
                xticklabels={\textbf{1000}, 0100, 0010, \textbf{0001}, \textbf{1100}, \textbf{1010}, \textbf{1001}, 0110, \textbf{0101}, {0011}, 1110, \textbf{1101}, \textbf{1011}, \textbf{0111}, \textbf{1111}},
        ytick={0, 0.2, 0.4, 0.6, 0.8, 1.0},
        xticklabel style={rotate=90, anchor=east},
        grid=both,
        grid style = {dashed}
    ]
\addplot[
        only marks,
        mark= x,
        color=darkgreen,line width=3pt,
        mark size=7.5pt
    ] coordinates {
        (1000, 0.652)
        (0100, 0.957)
        (0010, 0.913)
        (0001, 0.753)
        (1100, 0.531)
        (1010, 0.791)
        (1001, 0.483)
        (0110, 0.924)
        (0101, 0.61)
        (0011, 0.816)
        (1110, 0.732)
        (1101, 0.452)
        (1011, 0.684)
        (0111, 0.823)
        (1111, 0.465)

    };

\addplot[
	only marks,
	fill = red,
	mark=triangle,line width=1.5pt,
	color=red,
	mark size=6pt
	] coordinates {
        (1000, 0.843)
        (0100, 0.988)
        (0010, 0.989)
        (0001, 0.883)
        (1100, 0.754)
        (1010, 0.931)
        (1001, 0.697)
        (0110, 0.98)
        (0101, 0.776)
        (0011, 0.956)
        (1110, 0.92)
        (1101, 0.661)
        (1011, 0.841)
        (0111, 0.952)
        (1111, 0.752)

    };

\addplot[
        only marks,
        mark=triangle*,
        color=red,
        mark size=6pt
    ] coordinates {
        (1000, 0.651)
        (0100, 0.959)
        (0010, 0.934)
        (0001, 0.763)
        (1100, 0.535)
        (1010, 0.812)
        (1001, 0.481)
        (0110, 0.933)
        (0101, 0.617)
        (0011, 0.835)
        (1110, 0.777)
        (1101, 0.447)
        (1011, 0.694)
        (0111, 0.829)
        (1111, 0.493)

    };


\addplot[
        only marks,
        mark= + ,
        color=darkgreen,line width=3pt,
        mark size = 7.5pt
    ] coordinates {
        (1000, 0.862)
        (0100, 0.988)
        (0010, 0.987)
        (0001, 0.869)
        (1100, 0.797)
        (1010, 0.928)
        (1001, 0.756)
        (0110, 0.982)
        (0101, 0.807)
        (0011, 0.956)
        (1110, 0.927)
        (1101, 0.757)
        (1011, 0.865)
        (0111, 0.956)
        (1111, 0.765)

    };

\addplot[
        only marks,
        fill = blue,
        mark=o,line width=1.5pt,
        color=blue,
        mark size = 3.5pt
    ] coordinates {
        (1000, 0.867)
        (0100, 0.989)
        (0010, 0.989)
        (0001, 0.875)
        (1100, 0.803)
        (1010, 0.93)
        (1001, 0.762)
        (0110, 0.983)
        (0101, 0.822)
        (0011, 0.958)
        (1110, 0.932)
        (1101, 0.762)
        (1011, 0.867)
        (0111, 0.958)
        (1111, 0.763)

    };

\addplot[
        only marks,
        mark=*,
        color=blue,
        fill = blue,
        mark size = 3.5pt
    ] coordinates {
        (1000, 0.464)
        (0100, 0.953)
        (0010, 0.93)
        (0001, 0.416)
        (1100, 0.45)
        (1010, 0.783)
        (1001, 0.052)
        (0110, 0.922)
        (0101, 0.24)
        (0011, 0.817)
        (1110, 0.771)
        (1101, 0.052)
        (1011, 0.458)
        (0111, 0.795)
        (1111, 0.201)
        
    };


    \end{axis}
\end{tikzpicture}

\end{subfigure}
\begin{subfigure}{.5\textwidth}
\begin{tikzpicture}[scale=0.7]
    \begin{axis}[
    title={\large{$N=8$}},
        width=12cm, 
        height=8cm, 
         xlabel={Initial Genotype},
        ymin=0, ymax=1,
        xtick=data,
        symbolic x coords={1000, 0100, 0010, 0001, 1100, 1010, 1001, 0110, 0101, 0011, 1110, 1101, 1011, 0111, 1111},
                xticklabels={\textbf{1000}, 0100, 0010, \textbf{0001}, \textbf{1100}, \textbf{1010}, \textbf{1001}, 0110, \textbf{0101}, 0011, 1110, \textbf{1101}, \textbf{1011}, \textbf{0111}, \textbf{1111}},
        ytick={0, 0.2, 0.4, 0.6, 0.8, 1.0},
        xticklabel style={rotate=90, anchor=east},
        grid=both,
        grid style = {dashed}
    ]
    
    \addplot[
        only marks,
        mark= x,
        color=darkgreen,line width=3pt,
        mark size=7.5pt
    ] coordinates {
        (1000, 0.652)
        (0100, 0.957)
        (0010, 0.913)
        (0001, 0.753)
        (1100, 0.633)
        (1010, 0.839)
        (1001, 0.588)
        (0110, 0.944)
        (0101, 0.745)
        (0011, 0.861)
        (1110, 0.732)
        (1101, 0.452)
        (1011, 0.684)
        (0111, 0.823)
        (1111, 0.585) 
    };

    \addplot[
        only marks,
        fill = red,
        mark=triangle,line width=1.5pt,
        color=red,
        mark size=6pt
    ] coordinates {
        (1000, 0.843)
        (0100, 0.988)
        (0010, 0.989)
        (0001, 0.883)
        (1100, 0.839)
        (1010, 0.951)
        (1001, 0.795)
        (0110, 0.988)
        (0101, 0.876)
        (0011, 0.972)
        (1110, 0.92)
        (1101, 0.661)
        (1011, 0.841)
        (0111, 0.952)
        (1111, 0.779)

    };

    \addplot[
        only marks,
        mark=triangle*,
        color=red,
        mark size=6pt
    ] coordinates {
        (1000, 0.651)
        (0100, 0.959)
        (0010, 0.934)
        (0001, 0.763)
        (1100, 0.635)
        (1010, 0.854)
        (1001, 0.586)
        (0110, 0.953)
        (0101, 0.755)
        (0011, 0.878)
        (1110, 0.777)
        (1101, 0.447)
        (1011, 0.694)
        (0111, 0.829)
        (1111, 0.582)

    };

        \addplot[
        only marks,
        mark= + ,
        color=darkgreen,line width=3pt,
        mark size = 7.5pt
    ] coordinates {
        (1000, 0.862)
        (0100, 0.988)
        (0010, 0.987)
        (0001, 0.869)
        (1100, 0.86)
        (1010, 0.955)
        (1001, 0.818)
        (0110, 0.987)
        (0101, 0.854)
        (0011, 0.973)
        (1110, 0.927)
        (1101, 0.757)
        (1011, 0.865)
        (0111, 0.956)
        (1111, 0.84)

    };

    \addplot[
        only marks,
        fill = blue,
        mark=o,line width=1.5pt,
        color=blue,
        mark size = 3.5pt
    ] coordinates {
        (1000, 0.867)
        (0100, 0.989)
        (0010, 0.989)
        (0001, 0.875)
        (1100, 0.865)
        (1010, 0.957)
        (1001, 0.821)
        (0110, 0.989)
        (0101, 0.861)
        (0011, 0.975)
        (1110, 0.932)
        (1101, 0.762)
        (1011, 0.867)
        (0111, 0.958)
        (1111, 0.844)

    };

    \addplot[
        only marks,
        mark=*,
        color=blue,
        fill = blue,
        mark size = 3.5pt
    ] coordinates {
        (1000, 0.464)
        (0100, 0.953)
        (0010, 0.93)
        (0001, 0.416)
        (1100, 0.466)
        (1010, 0.811)
        (1001, 0.058)
        (0110, 0.95)
        (0101, 0.328)
        (0011, 0.879)
        (1110, 0.771)
        (1101, 0.052)
        (1011, 0.458)
        (0111, 0.795)
        (1111, 0.383)
        
    };
    \end{axis}
\end{tikzpicture}
\end{subfigure}

\end{figure}
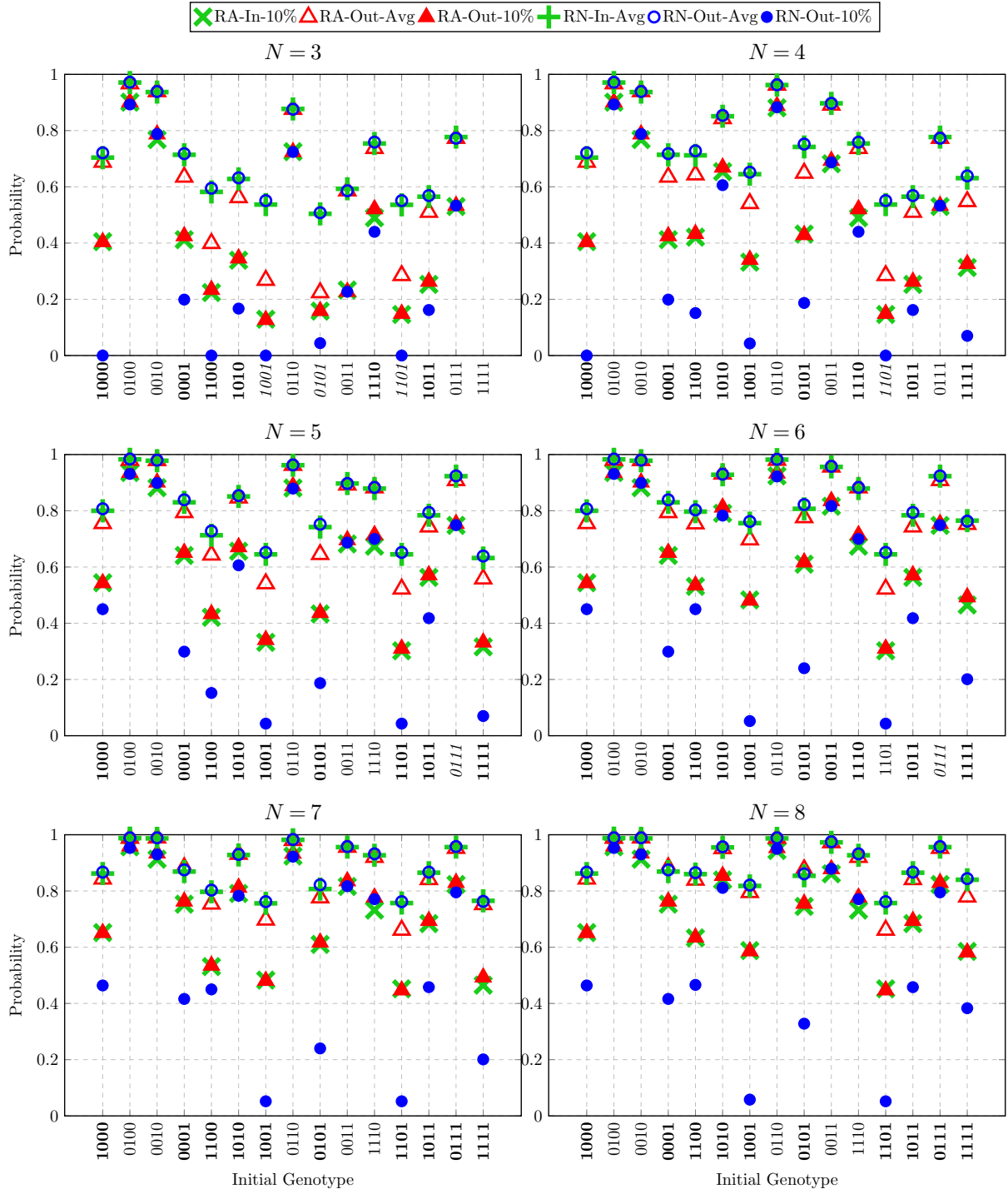

Finally, we compare the results of static and dynamic versions in terms of their compatibility with risk-aversion. We observe that risk-averse solutions are somewhat ``more successful" with the dynamic version as the number of \textit{bad} genotypes is much smaller and the sacrifice from the average performance becomes almost negligible compared to the gain from the worst 10\% performance as~$N$ increases. This suggests that risk-averse objective works much better with the dynamic version.

\section{Conclusions}
\label{sec:conclusions}

In this paper, we study the risk-averse Antibiotics Time Machine Problem. We consider two versions of the problem: static and dynamic. For both   versions, we propose scenario-based MILP formulations. Since these formulations do not scale well with scenario size, we develop scenario batch decomposition algorithms with  risk-averse subproblems. We introduce several algorithmic enhancements to further improve the computational efficiency.  We test the efficacy  of our approach on a real dataset. Our extensive computational experiments suggest that risk-averse solutions provide a significantly better worst-case performance compared to risk-neutral solutions, which comes with a deterioration from the average performance. We observe that the gains in terms of worst-case performance dominate the losses from the average performance as $N$ increase, and the benefits are even more significant for the dynamic version.

\subsubsection*{Funding}
This work was supported by the  Scientific and Technological Research Council of Turkey [grant number 120C151].

\bibliographystyle{informs2014} 
\bibliography{references}

\renewcommand{\theHchapter}{A\arabic{chapter}}

\begin{APPENDICES} 

\section{Matrix and Scenario Construction}
\label{app:matrixConstruction}

An important aspect of the Antibiotics Time Machine Problem is the calculation of transition probability  matrices for which we follow  \citet{mira2015rational}. 
Suppose that we are given a vector of growth rate measurements of genotypes for a specific antibiotic, denoted by
${\omega}\in\mathbb{R}_+^d$. Let us denote the set of neihgboring genotype pairs by $J$, and take a pair $(j,j')$ from the set $J$. If  ${\omega}_j'>{\omega}_j$, then the transition probability  from genotype $j$ to $j'$ is positive, otherwise the probability is zero. In our paper, we use the so-called 
\textit{correlated probability model}, in which the transition probabilities    are calculated as follows \citep{mira2015rational}:
\[
M_{j,j'} (\bm{\omega}) := \frac{ \max\{0, \omega_{j'} - \omega_{j} \} }{ \sum_{ j'':(j'',j) \in J}  \max\{0, \omega_{j''} - \omega_{j} \} }, \quad (j,j') \in J . 
\]




Now, suppose that the growth rate measurement for antibiotic $k$, denoted by ${\bm{\omega^k}}$, is a random variable and its $i$-th entry takes a value from the set $\{w_i^{k,r} : r=1,\dots,R\}$ with equal probability (see, e.g., \citet{mira2017statistical,mesum2024}). Assuming independence between entries, we conclude that ${\bm{\omega^k}}$ takes values from a set $ \{ {\bm{\omega^{k,s}}} : s\in\mathcal{S} \}$  for some index set   $\mathcal{S} \ :=\{1,\dots,R\}^d$.  To give some perspective, \citet{mira2017statistical} record the growth rates of $K=23$ antibiotics for $R=12$ times in which case there are $d=2^4$ genotypes. 
Due to the astronomically large size of the set $\mathcal{S}$, it is not  directly possible to incorporate all scenarios. Instead, we take a sample $\mathcal{H} \subseteq \mathcal{S}^K$, where $|\mathcal{H} |\ll | \mathcal{S}|^K$. 

\section{Effect of Enhancements for the Static Version with $N=4$}
\label{app:detailedComputations}

In order to compare the effect of enhancements, we solve the  risk-averse static version  for each initial genotype with $N=4$ and report the results in Table~\ref{tab:Static-N4-appendix}. 
Comparing Setting~0 and Setting~5, we see considerable savings in terms of computational time. In these experiments,  the symmetry breaking constraints and symmetry-enhanced Cartesian Cuts are the most effective enhancements followed by warm starting.

\begin{table}[h]
\centering
\small
\caption{Effect of enhancements for the risk-averse static version in terms of computational effort for $N=4$ (no iteration limit is enforced in this experiment). }\label{tab:Static-N4-appendix}
\begin{tabular}{|c|cr|cr|cr|cr|cr|cr|}
\hline
 & \multicolumn{2}{c|}{\textbf{Setting~0}} & \multicolumn{2}{c|}{\textbf{Setting~1}} & \multicolumn{2}{c|}{\textbf{Setting~2}} & \multicolumn{2}{c|}{\textbf{Setting~3}} & \multicolumn{2}{c|}{\textbf{Setting~4}} & \multicolumn{2}{c|}{\textbf{Setting~5}} \\ \hline
$\bm{i_I}$ & \multicolumn{1}{r}{\textbf{\#}} & \multicolumn{1}{r|}{\textbf{Time}} & \multicolumn{1}{r}{\textbf{\#}} & \multicolumn{1}{r|}{\textbf{Time}} & \multicolumn{1}{r}{\textbf{\#}} & \multicolumn{1}{r|}{\textbf{Time}} & \multicolumn{1}{r}{\textbf{\#}} & \multicolumn{1}{r|}{\textbf{Time}} & \multicolumn{1}{r}{\textbf{\#}} & \multicolumn{1}{r|}{\textbf{Time}} & \multicolumn{1}{r}{\textbf{\#}} & \multicolumn{1}{r|}{\textbf{Time}} \\ \hline
\textbf{1000} & 1 & 478 & 1 & 517 & 1 & 540 & 1 & 527 & 1 & 483 & 1 & 535 \\
\textbf{0100} & 1 & 107 & 1 & 229 & 1 & 128 & 1 & 231 & 1 & 88 & 1 & 235 \\
\textbf{0010} & 1 & 223 & 1 & 226 & 1 & 220 & 1 & 220 & 1 & 218 & 1 & 235 \\
\textbf{0001} & 4 & 2441 & 1 & 593 & 2 & 1246 & 1 & 694 & 2 & 1154 & 1 & 697 \\
\textbf{1100} & 2 & 1020 & 1 & 491 & 2 & 1039 & 1 & 547 & 2 & 1104 & 1 & 553 \\
\textbf{1010} & 1 & 323 & 1 & 307 & 1 & 299 & 1 & 281 & 1 & 282 & 1 & 274 \\
\textbf{1001} & 2 & 1068 & 2 & 1052 & 2 & 1121 & 2 & 1162 & 2 & 1193 & 2 & 1218 \\
\textbf{0110} & 1 & 302 & 1 & 331 & 1 & 321 & 1 & 314 & 1 & 261 & 1 & 318 \\
\textbf{0101} & 7 & 4599 & 3 & 1840 & 3 & 2187 & 3 & 2055 & 3 & 2094 & 3 & 2206 \\
\textbf{0011} & 2 & 590 & 2 & 662 & 2 & 590 & 2 & 649 & 2 & 644 & 2 & 625 \\
\textbf{1110} & 3 & 1226 & 2 & 768 & 2 & 742 & 2 & 743 & 2 & 684 & 2 & 734 \\
\textbf{1101} & 6 & 3156 & 3 & 1519 & 5 & 2927 & 3 & 1675 & 3 & 1566 & 3 & 1631 \\
\textbf{1011} & 3 & 852 & 2 & 570 & 2 & 647 & 2 & 636 & 2 & 619 & 2 & 627 \\
\textbf{0111} & 3 & 1251 & 2 & 848 & 2 & 822 & 2 & 804 & 2 & 816 & 2 & 852 \\
\textbf{1111} & 4 & 2163 & 3 & 1692 & 3 & 1708 & 3 & 1808 & 3 & 1771 & 3 & 1855 \\ \hline
\textbf{Avg} & 2.7 & 1320 & 1.7 & 776 & 2.0 & 969 & 1.7 & 823 & 1.9 & 865 & 1.7 & 840 \\ \hline
\end{tabular}
\end{table}

\end{APPENDICES}

\end{document}